\documentclass[a4paper]{amsart}
\usepackage[T1]{fontenc}
\usepackage[latin1]{inputenc}
\usepackage[backend=biber,style=numeric, doi=false,isbn=false,url=false, maxnames=10, minnames = 9]{biblatex}
\usepackage[shortlabels,inline]{enumitem}
\usepackage{amsfonts}
\usepackage{amssymb}
\usepackage{amsmath}
\usepackage{amsthm}
\usepackage[T1]{fontenc}
\usepackage[latin1]{inputenc}
\usepackage{faktor}
\usepackage[pdftex,x11names,dvipsnames]{xcolor}
\usepackage[pdftex]{graphicx}%
\usepackage{tikz}\usetikzlibrary{trees,matrix,arrows,calc,cd,decorations.pathmorphing,graphs, quotes, shapes.geometric, shapes.misc,positioning}
\usepackage[pdftex,linktocpage,breaklinks]{hyperref}
\hypersetup{
colorlinks=true,allcolors=blue,
linkcolor=blue,
citecolor=red,
anchorcolor=blue,
urlcolor=blue,
runcolor=red,
filecolor=red
pdfborder=0 0 0,
}
\PassOptionsToPackage{unicode}{hyperref}
\theoremstyle{plain}
\newtheorem{FactCounter}{dummy}[section]
\newtheorem{Theorem}[FactCounter]{Theorem} %
\newtheorem{Proposition}[FactCounter]{Proposition} %
\newtheorem{Lemma}[FactCounter]{Lemma} %
\newtheorem{Corollary}[FactCounter]{Corollary} %
\theoremstyle{definition}
\newtheorem{Definition}[FactCounter]{Definition} %
\theoremstyle{remark}
\newtheorem{Remark}[FactCounter]{Remark} %
\newtheorem{Example}[FactCounter]{Example} %
\newcommand\operator[1]{\mathop{\operatorname{#1}}\nolimits}
\newcommand{\id}{\operator{id}}
\newcommand{\dom}{\operator{dom}}
\newcommand{\im}{\operator{im}}
\newcommand{\source}{\textbf{s}}
\newcommand{\range}{\textbf{r}}
\newcommand{\HH}{{\mathcal{H}}}
\newcommand{\CC}{{\mathcal{C}}}
\newcommand{\muRelation}{\mathrel{\mu}}
\newcommand{\KS}{KS_{(G, \Gamma)}}
\newcommand{\TightIdeal}{\mathcal{I}_{\textup{tight}}}
\newcommand{\SingularIdeal}{\mathcal{I}_{\textup{sing}}}
\newcommand{\CKIdeal}{\mathcal{I}_{\textup{CK}}}
\newcommand{\SingularNotTightIdeal}{\SingularIdeal\setminus\TightIdeal}
\DeclareMathOperator{\Support}{\text{supp}}
\newcommand{\VertSet}{\Gamma^0}
\newcommand{\EdgeSet}{\Gamma^1}
\newcommand{\InvEdgeSet}{\Gamma^{-1}}
\newcommand{\PathSet}{\Gamma^*}

\begin{document}
\title{Simplicity of algebras and $C^*$-algebras of self-similar groupoids}
\author{Josiah Aakre}
\address{Department of Mathematics, University of Manchester, Manchester M13 9PL, United Kingdom}
\email{josiah.aakre@postgrad.manchester.ac.uk}

\begin{abstract}
Many previously studied path algebras and self-similar group algebras may be viewed as Steinberg algebras of self-similar groupoids. By way of inverse semigroup algebras, we characterize when the Steinberg algebra of a self-similar groupoid is simple. We show that the simplicity of the reduced $C^*$-algebra of a contracting self-similar groupoid coincides with the simplicity of the Steinberg algebra. As an aside, we show that simplicity of the two algebras sometimes depends only on the skeleton of the self-similar groupoid acting on a strongly connected graph. Finally, we apply our methods to examples including a self-similar groupoid akin to multispinal self-similar groups and a self-similar groupoid built from the well-known Basilica group.
\end{abstract}
\maketitle
\section{Introduction}
Self-similar groups were formally introduced by Nekrashevych in \cite{nekrashevych2005self}. This family of groups has been a fruitful source of examples with exotic properties such as the famous Grigorchuk group \cite{grigorchuk}, the first example of a group with intermediate growth. Self-similar groups come with an action on an infinite regular rooted tree such that the self-similarity of the tree is reflected in the action. Contracting self-similar groups form an important subclass in which the action of the group is ultimately dependent on a finite subset of the group. Self-similar groups have been used to build $C^*$-algebras \cite{Nekrashevych2009Alg} and $*$-algebras \cite{nekrashevych2015growthetalegroupoidssimple} which are often best understood in the contracting setting. Laca, Raeburn, Ramagge, and Whittaker \cite{SSGroupoidsWhitaker} introduced self-similar groupoids, a generalization of self-similar groups, in order to study dynamical properties of two associated $C^*$-algebras. In this work, we classify in the contracting setting the simplicity of the reduced $C^*$-algebra by way of its discrete analogue.

Several important classes of $C^*$-algebras admit a topological groupoid model. Often, these topological groupoids may be constructed and understood via inverse semigroups. See \cite{exel2008inversesemigroupscombinatorialcalgebras, Farthing2005,martinez2025algebraicsingularfunctionsdense, Paterson2002,STEINBERG2010689} for example. It was Steinberg \cite{STEINBERG2010689} who introduced the suitable discrete analogue of reduced $C^*$-algebras, which we now call Steinberg algebras. Steinberg algebras are constructed from ample groupoids, a class of topological groupoids that may always be constructed from inverse semigroups. It was later shown by Steinberg and Szak\'acs \cite{Steinbergandnora2020simplicityinversesemigroupetale} that just as ample groupoids may be constructed from inverse semigroups, Steinberg algebras may be constructed from inverse semigroup algebras. This common approach through topological groupoids to both the analytic and discrete settings suggests that the $C^*$-algebras and complex Steinberg algebras ought to be closely related. Indeed, there are a number of examples where the simplicity of complex Steinberg algebras and $C^*$-algebras coincide, including higher-rank graph $C^*$-algebras \cite{Robertson_2007} and the discrete analogue, Kumjian-Pask algebras \cite{clark2012groupoidgeneralizationleavittpath, pino2011kumjianpaskalgebrashigherrankgraphs} which include Leavitt path algebras.

Similarly, the simplicity of the $C^*$-algebras and complex Steinberg algebras coming from contracting self-similar groups has recently been shown to coincide \cite{brix2025hausdorffcoversnonhausdorffgroupoids, gardella2025simplicitycalgebrascontractingselfsimilar}. Note that the Steinberg algebra of a trivial self-similar group is just a Leavitt algebra. Our work extends these results to the setting of $C^*$-algebras and complex Steinberg algebras coming from contracting self-similar groupoids. The Steinberg algebras of self-similar groupoids are a large family of algebras encompassing both Leavitt path algebras \cite{Leavitt} (when the self-similar groupoid contains no nontrivial subgroups), Nekrashevych algebras \cite{nekrashevych2015growthetalegroupoidssimple} (when the self-similar groupoid is a group), and the algebraic analogues \cite{HazratPaskSimsSiera} of Exel-Pardo $C^*$-algebras \cite{EXELPardoAlgebras} (see \cite[Appendix A]{SSGroupoidsWhitaker}). By lifting to our setting the algorithm of Steinberg and Szak\'acs \cite{SteinbergandNora2023} along with the improvements of Gardella, Nekrashevych, Steinberg, and Vdovina \cite{gardella2025simplicitycalgebrascontractingselfsimilar}, we achieve what was previously achieved in both the graph algebra and self-similar group algebra settings: a description of the simplicity of the Steinberg algebra in terms of the defining ingredients, the self-similar groupoid and the graph being acted on.

We remark that in the $C^*$-algebra setting, Exel and Pardo's \cite{EXELPardoAlgebras} construction of algebras from self-similar groups acting on the path space of a graph (rather than just a bouquet of loops) gives a unified approach to both graph $C^*$-algebras and $C^*$-algebras coming from self-similar groups. Laca \textit{et al.} use self-similar groupoids to build a class of $C^*$-algebras which properly contain (see \cite[Appendix A]{SSGroupoidsWhitaker}) Exel and Pardo's class of $C^*$-algebras.

The paper is organized as follows. In Section \ref{Sec: Preliminaries}, we recall the preliminaries on self-similar groupoids, the associated inverse semigroup, and contracted inverse semigroup algebras. Section \ref{Sec: Congruence-free conditions} gives a complete characterization of when the inverse semigroup associated to a self-similar groupoid is congruence-free. We introduce the tight and singular ideals of the contracted inverse semigroup algebra in Section \ref{Sec: tight and singular ideals}. Section \ref{Sec: Collapsing Groupoids} reduces the question of simplicity to self-similar groupoids acting on strongly connected graphs. Section \ref{Sec: Supporting essnottight on groupoid subalgebra} shows that if the singular ideal is strictly larger than the tight ideal, then the difference in the ideals intersects the subalgebra spanned by the self-similar groupoid. In the contracting case, we show in Section \ref{Sec: Supporting ideals in nucleus} that the difference in the ideals intersects a subspace of the span of the nucleus. Section \ref{Sec: simplicity algorithm} gives a method to determine if a given contracting self-similar groupoid has a simple Steinberg algebra. Section \ref{Sec: simplicity of cstar algebras} couples the simplicity of the Steinberg algebra to the simplicity of the reduced $C^*$-algebra in the contracting case. Finally, we consider the simplicity of three examples in Section \ref{Sec: Examples}.

\section{Preliminaries}\label{Sec: Preliminaries}
Let $\Gamma=(\VertSet, \EdgeSet, \source, \range)$ be a row-finite directed graph with vertex set $\VertSet$, edge set $\EdgeSet$, and source and range maps $\source: \EdgeSet \rightarrow \VertSet$, and $\range: \EdgeSet\rightarrow \VertSet$. By row-finite, we mean that $\VertSet$ need not be finite, but for each $v\in\VertSet$, the set of edges leaving $v$, often denoted $\source^{-1}(v)$, is finite. We allow loops and multiple edges between two vertices. A path $p=e_1\ldots e_n$ in $\Gamma$ is a sequence of edges in $\EdgeSet$ such that $\range(e_i)=\source(e_{i+1})$ for $1\leq i < n$. We may represent the path $p$ pictorially by:
\[ \begin{tikzpicture}[>=Latex, edge from parent/.style={draw,-latex}]
\node (1) at (0,0) {};
\node (2) at (1, 0) {};
\node (3) at (2,0) {};
\node (4) at (3,0) {};
\node (5) at (4,0) {.};

\draw[->] (1) -- (2) node[above,pos=.5] {$e_1$};
\draw[->] (2) -- (3)node[above,pos=.5] {$e_2$};
\draw[-,dotted] (3) -- (4)node[above,pos=.5] {};
\draw[->] (4) -- (5)node[above,pos=.5] {$e_n$};
\end{tikzpicture}\]
A cycle is a path $p$ satisfying $\range(p)=\source(p)$ and $p$ is said to be nonempty if $p$ contains at least one edge. Note that our notion of paths coincides with the notion of walks in graph theory, as we allow a path to visit any given edge more than once. We denote the length of a path $p$ by $|p|$, and the set of all paths in $\Gamma$ of length $k$ is denoted by $\Gamma^k$. We consider vertices to be paths of length 0. The set of all finite length paths is denoted by $\PathSet$ and the source and range maps extend to $\PathSet$ by defining $\source(e_1\ldots e_n) = \source(e_1)$ and $\range(e_1\ldots e_n) = \range(e_n)$, or in the case of $v\in \VertSet$, we set $\source(v)=v=\range(v)$. Path concatenation is a natural partial operation on $\PathSet$: for paths $p=e_1\ldots e_n$ and $q=f_1\ldots f_m$, the concatenation $pq$ is the sequence $e_1\ldots e_nf_1\ldots f_m$ and is a path if and only if $\range(p)=\source(q)$. Equipping $\PathSet$ with this partial operation forms a category, whereas if we complete the operation by adjoining a zero to $\PathSet$ and set all undefined products to zero, $\PathSet$ forms a monoid. We note that both views are necessary in this paper; in particular, the categorical view is required in Section \ref{Sec: simplicity algorithm}. Vertices are idempotent under this operation. With the partial operation in mind, for each path $p$, we define
\[p\PathSet=\{pq: q\in \PathSet, \source(q)=\range(p)\}.\]
We say that two vertices $u$ and $v$ are strongly connected in $\Gamma$ if the sets $u\PathSet v$ and $v\PathSet u$ are each nonempty. We may similarly define an infinite path $\mathfrak{p}=e_1e_2e_3\ldots$ which has a source but no range, and we denote the set of all infinite paths as $\Gamma^{\omega}$. Let $\InvEdgeSet = \{e^{-1}: e\in \EdgeSet\}$ be a set of formal inverses of $\EdgeSet$ and extend the source and range maps by defining $\source(e^{-1}) = \range(e)$ and $\range(e^{-1}) = \source(e)$.

It is useful to identify $\PathSet$ with a directed graph with vertices corresponding to elements of $\PathSet$ and a directed edge from $p$ to $q$ if and only if $q=pe$ for some $e\in \EdgeSet$. It follows that for each $v\in \VertSet$, the set $v\PathSet$ is a directed tree rooted by $v$ and $\PathSet$ is a disjoint union of all such trees. If two trees $v\PathSet$ and $w\PathSet$ are isomorphic as directed graphs, we denote the set of all isomorphisms from $v\PathSet$ to $w\PathSet$ as $\text{Iso}(v,w)$. If $v=w$ then $\text{Iso}(v,v)$ is the automorphism group of $v\PathSet$ with identity denoted by $\text{id}_v$. Let $\tau$ be the equivalence relation on $\VertSet$ given by $(v,w)\in \tau$ if and only if $v\PathSet$ and $w\PathSet$ are isomorphic as directed graphs. Then the set
\[\textbf{Iso}(\PathSet) = \bigsqcup_{(v,w)\in \tau}\text{Iso}(v,w),\]
forms a groupoid with function composition as the partial operation defined precisely when composition is possible.

Let $G$ be a groupoid with object set $G^0=\VertSet$. A faithful action of $G$ on $\PathSet$ is an injective groupoid homomorphism $\sigma$ from $G$ to $\textbf{Iso}(\PathSet)$ such that $\sigma$ restricts to a bijection between $G^0$ and $\{\text{id}_v: v\in \VertSet\}$. For each $g\in G$, we identify $\sigma(g)$ and $g$. To indicate that $g\in \text{Iso}(v,w)$ we often write $\dom g = v\PathSet$ and $\im g = w\PathSet$. We denote the inverse of $g\in G$ by $g^{-1}$.

\begin{Definition}\label{Def: self-similar groupoid}
    Given a row-finite directed graph $\Gamma$, a \textit{self-similar groupoid} is a triple $(G, \sigma, \PathSet)$, often abbreviated to $(G, \Gamma)$ if the action is specified or implied, comprised of a groupoid $G$ together with a faithful action $\sigma$ of $G$ on $\PathSet$ such that for all $g\in G$ and all $e\in \EdgeSet\cap\dom g$, there exists some $h\in G$ such that
\[g(ep) = g(e)h(p),\]
for all $p\in \range(e)\PathSet$. Such an element $h$ is necessarily unique, and we call it the \textit{restriction} of $g$ to $e$, denoted $g|_e$. We may define the restriction of $g$ to a path $ep$ in $\dom g$ by recursively defining $g|_{ep} = (g|_{e})|_{p}$.
\end{Definition}
\begin{Remark}
    More details on self-similar groupoids can be found in \cite{SSGroupoidsWhitaker}. We note that our conventions differ from those of \cite{SSGroupoidsWhitaker} where a path is a sequence of edges $e_1\ldots e_n$ such that $\source(e_i)=\range(e_{i+1})$. This difference in convention has no bearing on the underlying concepts, and any examples or properties of self-similar groupoids which we cite from \cite{SSGroupoidsWhitaker} will be translated into our convention without explaining the translation.
    \end{Remark}
    The following properties of self-similar groupoids will be used throughout without reference.
    \begin{Proposition}[{\cite[Lemma 3.4, Proposition 3.6]{SSGroupoidsWhitaker}}]
    For all $g,h\in G$ with $\im g = \dom h$, and for all $p,q, w\in \PathSet$ with $\range(p)=\source(q)$, $pq\in \dom g$, and $w\in \dom g^{-1}$, we have
    \begin{enumerate}
        \item $\dom g|_p = \range(p)\PathSet$ and $\im g|_p = \range(g(p))\PathSet$,
        \item $\source(g(p)) = g(\source(p))$ and $\range(g(p)) = g|_p(\range(p))$,
        \item $\id_{\source(p)}|_p = \id_{\range(p)}$,
        \item $(hg)|_p = (h|_{g(p)})(g|_p)$,
        \item $g|_{pq}=(g|_p)|_q$,
        \item $g^{-1}|_w = (g|_{g^{-1}(w)})^{-1}$.
    \end{enumerate}
    \end{Proposition}
\begin{Definition}
    A self-similar groupoid $(G, \Gamma)$ is said to be \textit{contracting} if there exists a finite subset $A\subseteq G$ such that for all $g\in G$, there exists some natural number $n$ such that $g|_{p}\in A$ for all $p\in \Gamma^{n}$. As seen in \cite{brownloweWhittaker2023kkdualityselfsimilargroupoidactions}, there is a unique minimal set $N$ satisfying these properties called the \textit{nucleus} of $(G, \Gamma)$.
\end{Definition}
\begin{Remark}
    It is easy to see that if a self-similar groupoid $(G, \Gamma)$ is contracting, the finiteness of the nucleus implies a similar finiteness property in $\PathSet$: $\Gamma$ must contain a finite subgraph $\Omega$ such that every infinite path of $\Gamma$ contains finitely many edges from $\Gamma^1\setminus\Omega^1$. This observation does not play an explicit role in this paper, however, because the results of Sections \ref{Sec: Supporting ideals in nucleus}, \ref{Sec: simplicity algorithm} and \ref{Sec: simplicity of cstar algebras} apply only in the contracting setting, it is worth noting that not every row-finite graph admits a contracting self-similar groupoid action.
\end{Remark}

An inverse semigroup is a semigroup $S$ such that for each $s\in S$, there exists a unique $s^*\in S$ such that $ss^*s=s$ and $s^*ss^*=s^*$. Associated to any self-similar groupoid $(G, \Gamma)$, there is an inverse semigroup with zero denoted by $S_{(G, \Gamma)}$. Beware that the two uses of $*$ in $\Gamma^*$ and $s^*$ are unrelated.
\begin{Definition}\label{Def: Associated Inv SG}
    The inverse semigroup $S_{(G, \Gamma)}$ associated to a self-similar groupoid $(G, \Gamma)$ is generated by $\VertSet\cup \EdgeSet\cup \InvEdgeSet\cup G$ together with a zero $0$. We set $g^*=g^{-1}$ and $e^*=e^{-1}$ and $v^*=v$ for $g\in G$, $e\in \EdgeSet$, and $v\in \VertSet$. The operation of $S_{(G, \Gamma)}$ extends that of $G$ by defining any undefined products in $G$ as 0, and the following additional relations are imposed:
\begin{enumerate}
    \item $\source(e)e = e\range(e) = e$ for all $e\in \VertSet\cup \EdgeSet\cup \InvEdgeSet$;
    \item $uv=0$ for distinct $u,v\in \VertSet$;
    \item $e^* f=0$ for distinct $e,f\in \EdgeSet$;
    \item $e^* e= \range(e)$ if $e\in \EdgeSet$;
    \item $ge = g(e)g|_e$ for all $g\in G$ and $e\in \EdgeSet\cap \dom g$;
    \item $\id_v = v$ for all $v\in \VertSet$ and $\id_v\in G$.
\end{enumerate}
\end{Definition}
This is not a minimal list of relations. We highlight that the first four relations come from what is called the graph inverse semigroup of $\Gamma$ \cite{AshHallGraphISGs}, while the last two relations incorporate the action of $G$.

We define an action of $S_{(G, \Gamma)}$ on $\PathSet$ by allowing $G\subseteq S_{(G, \Gamma)}$ to act by the given self-similar action, and defining actions of the remaining generators as follows. An edge $e\in \EdgeSet \subseteq S_{(G, \Gamma)}$ acts as a bijection $\range(e)\PathSet\rightarrow e\PathSet$ by concatenation: $p\mapsto ep$. The inverse edge $e^*\in \InvEdgeSet\subseteq S_{(G, \Gamma)}$ acts as the inverse map $e\PathSet \rightarrow \range(e)\PathSet$, given by $ep\mapsto p$. Finally, the zero map acts with empty domain. It is routine to check that this defines an action of $S_{(G, \Gamma)}$ on $\PathSet$ by partial bijections.

Similar to the self-similar group case \cite{LawsonPerrotMonoid}, it can be shown that the nonzero elements of $S_{(G, \Gamma)}$ may be canonically expressed in the form $pgq^*$ where $p,q\in \PathSet$ and $g\in \text{Iso}(\range{(q)},\range{(p)})$. One may check that $pgq^*$ acts bijectively with domain $q\PathSet$ and image $p\PathSet$. It is easy to verify that the semilattice of idempotents of $S_{(G, \Gamma)}$ is the subset
\[E_{(G, \Gamma)}=\{pp^*: p\in \PathSet\}\cup\{0\}.\] 

\begin{Definition}\label{Def: Contracted Semigroup Algebra}
    For a field $K$ and an inverse semigroup $S$ with zero, define the contracted inverse semigroup algebra $KS$ as the $K$-algebra with basis $S^\sharp =S\setminus\{0\}$ and multiplication extending that of $S$ where we identify the zeros of $S$ and $K$.
\end{Definition}
Fix a row-finite directed graph $\Gamma$ and a self-similar groupoid $(G, \Gamma)$. We are interested in studying the contracted algebra of $S_{(G, \Gamma)}$. The objective of this paper is to classify the simplicity of a quotient of $\KS$ by what is known as the tight ideal. It was originally observed by Munn \cite{munn1979simple} that for an inverse semigroup $S$, a proper quotient of $S$ gives rise to a proper quotient of $KS$. Thus, we must understand quotients of $S$ before we understand quotients of $KS$. 

We briefly recall a few definitions standard in inverse semigroup theory. Let $S$ be an inverse semigroup with a zero and let $E(S)$ be the semilattice of idempotents of $S$. Let $\muRelation$ be a relation on $S$ given by $a\muRelation b$ if and only if $aea^*=beb^*$ for all $e\in E(S)$. It turns out that $\muRelation$ is a congruence, that is, an equivalence relation which is compatible with the operation of the semigroup. We say that $S$ is \textit{fundamental} (see \cite{howie1995fundamentals, Munn_Fundamental_ISG}) if $a=b$ whenever $a\muRelation b$. We say that $S$ is \textit{0-simple} if the only ideals of $S$ are $S$ and $\{0\}$. Finally, we say that the semilattice of idempotents $E(S)$ is \textit{0-disjunctive} (see \cite{lawson2020primer}) if for all $e,f\in E(S)$, there exists some $h\in E(S)$ such that of the products $eh$ and $fh$, exactly one is zero.
An inverse semigroup with zero is \textit{congruence-free} if it admits no nontrivial proper quotients. The following is a well-known result of inverse semigroup theory due to \cite[Theorem 2.2]{Trotter1974} and independently, {\cite[Theorem 5]{Baird_1975}}.
\begin{Theorem}\label{Thm: CongruenceFreeConditions}
    Let $S$ be an inverse semigroup with zero. Then $S$ is congruence-free if and only if $S$ is fundamental, $S$ is 0-simple, and $E(S)$ is 0-disjunctive.
\end{Theorem}
\section{Congruence-free associated inverse semigroups}\label{Sec: Congruence-free conditions}
It is shown in \cite{LawsonPerrotMonoid} that $S_{(G, \Gamma)}$ is congruence-free for every self-similar group $(G, \Gamma)$. The self-similar groupoid setting does not allow such a convenience. In this section, we give necessary and sufficient conditions on the self-similar groupoid $(G, \Gamma)$ so that $S_{(G, \Gamma)}$ is congruence-free. Theorem \ref{Thm: CongruenceFreeAppliedS_G} first appeared in the author's master's dissertation.

\begin{Theorem}\label{Thm: CongruenceFreeAppliedS_G}
    Let $(G, \Gamma)$ be a self-similar groupoid and let $S_{(G, \Gamma)}$ be the associated inverse semigroup. Then $S_{(G, \Gamma)}$ is congruence-free if and only if the following hold:
    \begin{enumerate}
        \item  $|\source^{-1}(v)|\neq 1$ for all $v\in \VertSet$,
        \item  for each pair of distinct vertices $v,w \in\VertSet$, there exists $p\in w\PathSet$ and $q\in v\PathSet$ such that $G$ intersects $\text{Iso}(\range(p), v)$ and $\text{Iso}(\range(q), w)$.
    \end{enumerate}
\end{Theorem}
Theorem \ref{Thm: CongruenceFreeAppliedS_G} is a result of the remainder of the section.
\begin{Proposition}\label{Prop: SGG always fundamental}
    Let $(G, \Gamma)$ be a self-similar groupoid and let $S_{(G, \Gamma)}$ be the associated inverse semigroup. Then $S_{(G, \Gamma)}$ is fundamental.
\end{Proposition}
\begin{proof}
    First, observe that no nonzero element $s$ is $\muRelation$ related to $0$ as $ss^*ss^* = ss^*\neq 0s^*s0$. Consider two nonzero elements of $S_{(G, \Gamma)}$ in canonical form, $a=pgq^*$ and $b=whx^*$ with $h,g\in G$ and $p,q,w,x\in \PathSet$. Suppose that $a \muRelation b$. It is well-known that this implies that $aa^*\muRelation bb^*$ and $a^*a\muRelation b^*b$. For each nonzero $s\in S_{(G, \Gamma)}$, there is exactly one $v\in \VertSet$ such that $ss^*v ss^*=ss^*$, and $ss^*uss^*=0$ for all other $u\in \VertSet$. It follows that $aa^*=bb^*$ and $a^*a=b^*b$. We then have,
    \begin{align*}
        aa^*=bb^* &\iff pgq^*qg^*p^* = whx^*xh^*w^*\\
        &\iff pg\range(q)g^*p^* = wh\range(x)h^*w^*\\
        &\iff pg\id_{\range(q)}g^*p^* = wh\id_{\range(x)}h^*w^*\\
        &\iff pp^*=ww^*,
    \end{align*}
    and similarly $qq^*=xx^*$. This clearly implies that $p=w$ and $q=x$.

    Because $\muRelation$ is a congruence, $a\muRelation b$ implies $p^*aq\muRelation p^*bq$, however $p^*aq = p^*pgq^*q= \range(p)g\range(q) = g$, and similarly $p^*bq = h$. Since $g\muRelation h$, we have that $geg^*=heh^*$ for each idempotent $e\in S_{(G, \Gamma)}$. When $e=0$, equality is trivial, but when $e\neq 0$, $e$ takes the form $zz^*$ for some $z\in \PathSet$. Observe that when $gz\neq 0$, we have, 
    \begin{align*}
    gzz^*g^* &= gz(gz)^* \\
    &= g(z)g|_z(g(z)g|_z)^* \\
    &= g(z)g|_zg|_z^*g(z)^* \\
    &= g(z)g(z)^*.
    \end{align*}
    Applying the same simplification to $hzz^*h^*$, we see that $g\muRelation h$ implies $g(z)g(z)^* = h(z)h(z)^*$, and therefore $g(z)=h(z)$. Observe also that $gz\neq 0$ if and only if $hz\neq 0$, so the action of $g$ and $h$ agree for all $z$ in $\dom g = \dom h$. By the faithfulness of the action, $g=h$. We conclude that $a\mathrel{\muRelation} b$ implies $a=b$ and so $S_{(G, \Gamma)}$ is fundamental.
\end{proof}
Let us now work towards characterizing when $S_{(G, \Gamma)}$ is 0-simple. Recall the action of $S_{(G, \Gamma)}$ on $\PathSet$.
\begin{Lemma}\label{Lem: half of J-relation argument}
    Let $S=S_{(G, \Gamma)}$. For distinct $v,w\in \VertSet$, the following are equivalent:
    \begin{enumerate}
        \item[(1)] $v\in SwS$;
        \item[(2)] there exists a map $s\colon v\PathSet\rightarrow w\PathSet$ in $S$;
        \item[(3)] there exists a path $p\in v\PathSet$ such that $G\cap \text{Iso}(\range(p),w)$ is nonempty.
    \end{enumerate}
\end{Lemma}
\begin{proof}
    First, suppose $(2)$. Then $v=s^* w s$ because $w$ acts as the identity on $w\PathSet$. Thus, $(2)\implies (1)$

    Now suppose $(1)$. Then there exists some $s,t\in S$ such that $v=swt$. Observe that
    \[v\PathSet = \dom v = \dom swt \subseteq \dom t. \]
    Suppose $t=agb^*$ with $a,b\in \PathSet$ and $g\in G$. Then $\dom t = b\PathSet$. Since $v\PathSet$ is maximal (with respect to set inclusion) among all sets of the form $q\PathSet$ for any $q\in \PathSet$, we must have that $b=v$, and so $\dom t = v\PathSet$.
    
    Now because $v=swt$ is not the zero map, $\im t$ must intersect $\im w = w\PathSet$. But as $t=agv$, we have $\im t = a\PathSet$, and therefore $a\in w\PathSet$. It follows that $t$ is a map from $v\PathSet$ into $w\PathSet$, thus, $(1)\implies (2)$.

    In $t=agv$ above, observe that $a\in w\PathSet$ and $g: \range(a)\PathSet\rightarrow v\PathSet$. So $(1)\implies (3)$. Finally, $(3)$ clearly implies $(2)$.
\end{proof}
\begin{Proposition}\label{Prop: 0-simple classified}
    Let $(G, \Gamma)$ be a self-similar groupoid and let $S=S_{(G, \Gamma)}$ be the associated inverse semigroup. Then $S$ is 0-simple if and only if for each pair of distinct vertices $v,w \in\VertSet$, there exist $p\in w\PathSet$ and $q\in v\PathSet$ such that $G$ intersects $\text{Iso}(\range(p), v)$ and $\text{Iso}(\range(q), w)$.
\end{Proposition}
\begin{proof}
    It is an elementary fact of inverse semigroup theory that every principal ideal may be generated by an idempotent element. In our context, we can in fact generate every nonzero principal ideal by an idempotent in $\VertSet$, that is, a vertex. To see this, let $p\in \PathSet$ so that $pp^*$ and $p^*p$ are idempotents. Observe that $pp^* = pp^*pp^*$ and $p^*p=p^*pp^*p$ and therefore $pp^*$ and $p^*p$ generate the same principal ideals. By the defining relations of $S$, $p^*p=\range(p)\in \VertSet$, and therefore every principal ideal may be generated by an idempotent from $\VertSet$.

    Now, $S$ is 0-simple if and only if $SuS=SvS$ for each $u,w\in \VertSet$. By Lemma $\ref{Lem: half of J-relation argument}$, this occurs if and only if there exists $p\in w\PathSet$ and $q\in v\PathSet$ such that $G$ intersects $\text{Iso}(\range(p), v)$ and $\text{Iso}(\range(q), w)$, establishing the claim.
\end{proof}
Consider the following example which illustrates how both the connectivity of $\Gamma$ and the maps in $G$ play a role in $0-$simplicity.
\begin{Example}
    Let $\Gamma$ be the disjoint union of two directed 2-cycles as shown. In this case, $\text{Iso}(i,j)$ is a singleton $\{f_{i,j}\}$ say, for each $i,j\in \{u,v,x,y\}$. Under this notation, $\textbf{Iso}(\Gamma^*)=\{f_{i,j}: i,j\in \{u,v,x,y\}\}$ Define the groupoid,
    \[G=\{\id_u, \id_v, \id_x, \id_y, f_{u,x}, f_{x,u}, f_{v,y}, f_{y,v}\}.\]
    It is easy to verify that $G$ is a self-similar groupoid.
    \begin{figure}[ht!]
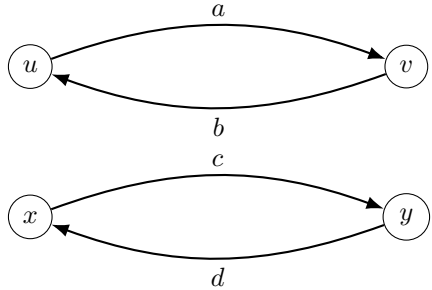

        \centering
        \tikz [>=Latex]{
        \graph [edge quotes={auto}]{ 
  "$u$" [at={(0,1)}, shape=circle, draw] ->[bend left=20,"$a$", thick]  "$v$" [at={(4,1)}, shape=circle, draw];
  "$x$" [at={(0,0)}, shape=circle, draw] ->[bend left=20,"$c$", thick]  "$y$" [at={(4,0)}, shape=circle, draw];
  "$v$" ->[bend left=20, "$b$", thick] "$u$";
  "$y$" ->[bend left=20, "$d$", thick] "$x$";
  
};}
        \caption[A]{The graphs $\Gamma$ and $\PathSet$.}
        \label{fig: double 2-cycle graph}
    \end{figure}
    We may draw a directed graph to represent which elements of $\textbf{Iso}(\PathSet)$ lie in $G$ as shown in Figure \ref{fig: connectivity of double 2 cycle Groupoid}. We are particularly interested in the non-idempotent elements, so the loops are omitted.
    \begin{figure}[ht!]
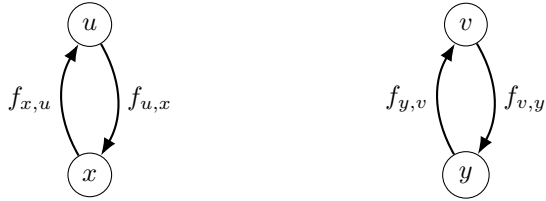

        \centering
        \tikz [>=Latex]{
        \graph [edge quotes={auto}]{ 
  "$u$" [at={(0,1)}, shape=circle, draw] ->[bend left,"$a$", thick, white]  "$v$" [at={(4,1)}, shape=circle, draw];
  "$x$" [at={(0,0)}, shape=circle, draw] ->[bend left,"$c$", thick, white]  "$y$" [at={(4,0)}, shape=circle, draw];
  "$v$" ->[bend left, "$b$", thick, white] "$u$";
  "$y$" ->[bend left, "$d$", thick, white] "$x$";

  "$u$" ->[bend left, "$f_{u,x}$", thick]  "$x$";
  "$v$" ->[bend left, "$f_{v,y}$", thick]  "$y$" ;
  "$x$" ->[bend left, "$f_{x,u}$", thick] "$u$";
  "$y$" ->[bend left, "$f_{y,v}$", thick] "$v$";
  
};}
        \caption[A]{Non-idempotent elements of $G$.}
        \label{fig: connectivity of double 2 cycle Groupoid}
    \end{figure}
    Now consider the inverse semigroup $S=S_{(G, \Gamma)}$ associated to the self-similar groupoid $(G, \Gamma)$. By direct calculation, we may verify that the vertices $u,v,x,y$ generate the same principal ideals of $S$. Observe that,
    \begin{align*}
        v&=a^*ua,\\
        u&=b^*vb,\\
        y&=c^*xc,\\
        x&=d^*yd,\\
    \end{align*}
    and therefore $SuS=SvS$ and $SxS=SyS$. Further observe that,
    \begin{align*}
        u&=f_{x,u}xf_{u,x},\\
        x&=f_{u,x}uf_{x,u},\\
        v&=f_{y,v}yf_{v,y},\\
        y&=f_{v,y}vf_{y,v},\\
    \end{align*}
    and therefore $SuS=SxS$ and $SvS=SyS$. As discussed in the proof of Proposition \ref{Prop: 0-simple classified}, the 0-simplicity of $S$ depends only on the principal ideals of the vertices of $\Gamma^0$, and therefore $S$ is 0-simple.
\end{Example}
\begin{Proposition}\label{Prop: SGG 0disjunctive iff Gamma out-degree is never 1}
    Let $(G, \Gamma)$ be a self-similar groupoid and let $S_{(G, \Gamma)}$ be the associated inverse semigroup. The semilattice $E_{(G, \Gamma)}$ is 0-disjunctive if and only if $|\source^{-1}(v)|\neq 1$ for all $v\in \VertSet$.
\end{Proposition}
\begin{proof}
    First, suppose that $|\source^{-1}(v)|\neq 1$ for all $v\in \VertSet$. Let $pp^*$ and $qq^*$ be distinct nonzero idempotents of $S_G$ with $p,q\in \PathSet$. Without loss of generality,  either $p$ is a prefix of $q$, or neither $p$ nor $q$ is a prefix of the other.
    
    First, suppose that $p$ is a prefix of $q$. Because $p\neq q$, we may write $q=pw$ for some $w\in \PathSet\setminus\VertSet$, and so the out-degree of $\range(p)$ must not be zero. Let $w_1$ be the first edge of $w$. By the assumption on out-degree, there is some other edge $a$ in $\range(p)\EdgeSet$ in addition to $w_1$. Consider the idempotent $(pa)(pa)^*$. Clearly $pp^*(pa)(pa)^*=(pa)(pa)^*\neq 0$. On the other hand,
    \begin{align*}
        qq^*(pa)(pa)^* &= (pw)(pw)^*(pa)(pa)^*\\
        &= pww^*p^*paa^*p^*\\
        &= pp^*pww^*aa^*p^*\\
        &= pww^*w_1w_1^*aa^*p^*\\
        &=0
    \end{align*}
    because $w_1\neq a$.

    Now suppose neither $p$ nor $q$ is a prefix of the other.
    Then $pp^*pp^* = pp^*$ but $qq^*pp^*=0$, because $q^*p=0$ by relations 1., 3., and 4.\ of Definition \ref{Def: Associated Inv SG}. This establishes the reverse implication.

    We prove the forward implication by contrapositive. Suppose that there exists some $v\in \VertSet$ such that $|\source^{-1}(v)|=1$. Let $e$ be the unique edge leaving $v$, so $v$ and $ee^*$ are two distinct idempotents in $E_{(G, \Gamma)}$. Let $q\in \PathSet$. Observe that $qq^*=\source(q)qq^*$ and so $vqq^*=0$ if and only if $\source(q)\neq v$. Similarly $ee^*qq^* = ee^*\range(e^*)\source(q)qq^* = ee^*v\source(q)qq^*$ and so $ee^*qq^*=0$ if and only if $\source(q)\neq v$. Therefore there is no idempotent $h$ for which exactly one of the products $vh$ and $ee^*h$ is zero, and so the semilattice $E_{(G, \Gamma)}$ is not 0-disjunctive, completing the proof.
\end{proof}
\begin{proof}[Proof of Theorem \ref{Thm: CongruenceFreeAppliedS_G}.]
Direct application of Theorem \ref{Thm: CongruenceFreeConditions} along with Propositions \ref{Prop: SGG always fundamental}, \ref{Prop: 0-simple classified}, and \ref{Prop: SGG 0disjunctive iff Gamma out-degree is never 1} establishes the result.
\end{proof}
Up to this point, our only assumption regarding $\Gamma$ is row-finiteness. To avoid trivialities, it is reasonable to demand that $\Gamma$ contains at least one edge. Under this assumption, it follows from Lemma \ref{Lem: half of J-relation argument} and Proposition \ref{Prop: 0-simple classified} that every vertex of $\Gamma$ must emit at least one edge in order for the associated inverse semigroup to be 0-simple. With this in mind, the following corollary to Theorem \ref{Thm: CongruenceFreeAppliedS_G} is immediate.
\begin{Corollary}\label{Cor: CongruenceFreeAppliedS_G with nonempty graph}
    Let $(G, \Gamma)$ be a self-similar groupoid with $\Gamma^1$ nonempty and let $S_{(G, \Gamma)}$ be the associated inverse semigroup. Then $S_{(G, \Gamma)}$ is congruence-free if and only if the following hold:
    \begin{enumerate}
        \item  $|\source^{-1}(v)|> 1$ for all $v\in \VertSet$,
        \item  for each pair of distinct vertices $v,w \in\VertSet$, there exists $p\in w\PathSet$ and $q\in v\PathSet$ such that $G$ intersects $\text{Iso}(\range(p), v)$ and $\text{Iso}(\range(q), w)$.
    \end{enumerate}
\end{Corollary}
We call any self-similar groupoid which satisfies the conditions 1.\ and 2.\ of Corollary \ref{Cor: CongruenceFreeAppliedS_G with nonempty graph} a self-similar groupoid of type (CF).
\section{The tight and singular ideals of \texorpdfstring{$\KS$}{KS}}\label{Sec: tight and singular ideals}
Let $S$ be an inverse semigroup with zero and let $E(S)$ be the semilattice of idempotents of $S$. It is standard to consider the natural partial order on $E(S)$ given by $e\leq f$ if and only if $e=gf$ for some $g\in E(S)$. For a nonzero idempotent $e\in E(S)$, we say that $F \subseteq E(S)$ covers $e$ if $f \leq e$ for all $f \in F$, and for all nonzero $g \leq e$, there exists some $f \in F$ such that $gf \neq 0$. Let $KS$ be the contracted $K$-algebra of $S$. The tight ideal of $KS$, denoted by $\TightIdeal$, is the ideal generated by
all products of the form
\[\prod_{f\in F}(e-f),\]
where $e \in E(S)$, and $F \subseteq E(S)$ is a finite cover of $e$. See \cite{Steinbergandnora2020simplicityinversesemigroupetale} for details.
    
The so-called \textit{Cuntz-Krieger} relations were introduced in \cite{Cuntz1980ACO} and have since been widely used to define various graph algebras and related objects, see for example \cite{Leavitt, HazratPaskSimsSiera, SSGroupoidsWhitaker, Nekrashevych2009Alg}. For a self-similar groupoid $(G, \Gamma)$, these relations are used to derive an ideal of $\KS$ called the Cuntz-Krieger ideal and denoted here by $\CKIdeal$. The ideal $\CKIdeal$ is generated by the elements
    \[v-\sum_{e\in v\EdgeSet}ee^*,\]
    for each $v\in\VertSet$. We show in Proposition \ref{Prop: tight hands on descriptions} that the Cuntz-Krieger ideal is exactly the tight ideal. To show that the Cuntz-Krieger ideal and the tight ideal coincide, we need the following standard observation about the Cuntz-Krieger ideal with proof included for completeness.
\begin{Proposition}\label{Prop: induction on CKideal generators}
    Let $(G, \Gamma)$ be a self-similar groupoid. For any vertex $v\in\VertSet$ and any natural number $n$, the Cuntz-Krieger ideal contains ${v-\sum_{p\in v\Gamma^n}pp^*}$.
\end{Proposition}
\begin{proof}
    Let $v\in\VertSet$. For every integer $n\geq 0$, let $R_{v}^n=\{\range(p): p\in v\Gamma^n \}$ be the vertices reachable from $v$ using a path of length $n$. This set is finite because $\Gamma$ is row-finite. Observe that $\CKIdeal$ contains
    \[\sum_{p\in v\Gamma^n}p\left(\sum_{u\in R_{v}^n}\left(u-\sum_{e\in u\EdgeSet}ee^*\right)\right)\sum_{p\in R_v^{n}}p^* = \sum_{p\in v\Gamma^n}pp^*-\sum_{q\in v\Gamma^{n+1}}qq^*.\]
    Denote the right side of the above equation by $a_n$. It follows that $\CKIdeal$ contains,
    \[\sum_{n=0}^{n-1}a_n = v-\sum_{p\in v\Gamma^n}pp^*.\]
\end{proof}
Fix a row-finite directed graph $\Gamma$ and a self-similar groupoid $(G, \Gamma)$, and continue to denote the associated inverse semigroup by $S_{(G, \Gamma)}$ and its semilattice of idempotents by $E_{(G, \Gamma)}$. We now obtain a practical description of the tight ideal of the contracted inverse semigroup algebra of a self-similar groupoid.
\begin{Proposition}\label{Prop: tight hands on descriptions}
	Let $a\in\KS$. The following are equivalent:
	\begin{enumerate}
		\item[(1)] $a\in \CKIdeal$;
		\item[(2)] $a \in \TightIdeal$;
		\item[(3)] every $\mathfrak{p} \in \Gamma^{\omega}$ has a prefix $p \in \PathSet$ satisfying $ap = 0$;
        \item[(4)] for all $v\in \VertSet$, there exists some $n$ such that $av\Gamma^n=0$.
	\end{enumerate}
\end{Proposition}
\begin{proof}
	We prove equivalence in the following order: $(2)\implies (3)\implies (4) \implies (1) \implies (2)$
    
    $(2)\implies (3)$: We show that elements satisfying $(3)$ form an ideal containing $\TightIdeal$. Let $a,b\in\KS$ satisfy (3) and let $\mathfrak{w}\in\Gamma^\omega$. By assumption, there exist prefixes $w_a$ and $w_b$ of $\mathfrak{w}$ such that $aw_a=0=bw_b$. If $w$ is the longer of the two prefixes, clearly $(a+b)w=0$. For any $s\in S_{(G, \Gamma)}$, we have $saw=0$, so elements satisfying (3) form a left ideal. 
    
    Now observe that the action of $S_{(G, \Gamma)}$ on $\PathSet$ naturally extends to an action on $\Gamma^\omega$. Suppose $s$ maps $\mathfrak{w}$ to $\mathfrak{w'}$. By assumption, there exists a prefix $z'$ of $\mathfrak{w'}$ such that $az'=0$. Let $z$ be the prefix of $\mathfrak{w}$ that $s$ maps to $z'$. Then there is some $s'\in S_{(G, \Gamma)}$ such that $asz=az's'=0$. Thus, elements satisfying (3) form a two-sided ideal.
    
    Consider an arbitrary generator of $\TightIdeal$, 
    \[\prod_{qq^*\in F}(pp^*-qq^*),\]
    where $p \in \PathSet$, and $F \subseteq E_{(G, \Gamma)}$ is a finite cover of $pp^*$. Observe that by the definition of a cover, $qq^*\leq pp^*$ and consequently $p$ is a prefix of $q$ for each $qq^*\in F$. Let $n=\max\{|q|: qq^*\in F\}$ and let $w$ be the length $n$ prefix of some $\mathfrak{w}\in \Gamma^\omega$. If $w\notin p\PathSet$, then $pp^*w=0$ and so $qq^*w = qq^*pp^*w=0$ for all $qq^*\in F$. If $w\in p\PathSet$, then $ww^*\leq pp^*$ so $qq^*w\neq 0$ for some $qq^*\in F$. By assumption, $|w|\geq |q|$, and so $w\in q\PathSet$. It follows that,
    \[(pp^*-qq^*)w=pp^*w-qq^*w=w-w=0.\]
    In each case, $\prod_{qq^*\in F}(pp^*-qq^*)w=0$ and so the generators of $\TightIdeal$ satisfy (3) yielding $(2)\implies (3)$.

    $(3)\implies (4)$: By contrapositive, suppose that (4) fails for an element $a\in\KS$, and consequently, there exists some $v\in\VertSet$ with infinitely many $p\in v\PathSet$ such that $ap\neq 0$. Let $T$ denote the subgraph of $v\PathSet$ containing only those vertices $p$ such that $ap\neq 0$, and edges between such vertices. Since $ap\neq 0$ implies that $aq\neq 0$ for any prefix $q$ of $p$, we have $T$ is an infinite subtree of $v\PathSet$ containing the root $v$. Since $\Gamma$ is row-finite, it follows from K\H onig's lemma that $T$ contains an infinite path $\mathfrak{q}$ such that $aq\neq 0$ for all prefixes $q$ of $\mathfrak{q}$, and so $(3)$ fails. Therefore by contrapositive, $(3) \implies (4)$.
    
    $(4)\implies (1)$: Now suppose $a\in \KS$ satisfies (4). Because $a$ has finite support on the basis of $\KS$, there exist finitely many vertices $v$ such that $av\neq 0$. Let $V_a$ be this set of vertices. For each $v\in V_a$, let $n_v$ be the integer such that $ap=0$ for all $p\in v\Gamma^{n_v}$.
    
    By Proposition \ref{Prop: induction on CKideal generators}, the Cuntz-Krieger ideal $\CKIdeal$ contains the element
    \[v-\sum_{p\in\Gamma^{n_v}}pp^*,\]
    for all $v\in V_a$, and therefore $\CKIdeal$ also contains the sum of these elements. It follows that $\CKIdeal$ also contains,
    \[a\left(\sum_{v\in V_a}\left(v-\sum_{p\in\Gamma^{n_v}}pp^*\right)\right)=a.\]
    Thus, $(4)\implies (1)$.
    
    $(1)\implies (2)$: It follows from relations 1.\ and 3.\ of Definition \ref{Def: Associated Inv SG} that the generators of $\CKIdeal$ may be factored,
    \[v-\sum_{e\in v\EdgeSet}ee^* = \prod_{e\in v\EdgeSet}(v-ee^*).\]
    Since $\{ee^*: e\in v\EdgeSet\}$ is a finite cover for $v$, this factorization is a generator of $\TightIdeal$, verifying that $(1)\implies (2)$. This completes the proof.
    \end{proof}
For a congruence-free inverse semigroup with zero $S$, another ideal of $KS$ was introduced in \cite{Steinbergandnora2020simplicityinversesemigroupetale}. An element $a \in KS$
is \textit{singular} if, for all nonzero $e\in E(S)$, there exists a nonzero $f \leq e$ with $af = 0$. It is shown that the singular elements form a two-sided ideal containing $\TightIdeal$ called the \textit{singular ideal}. This ideal has also been called the essential ideal in \cite{ExelEssential_2022, gardella2025simplicitycalgebrascontractingselfsimilar}, for example. Denote the singular ideal by $\SingularIdeal$. The following theorem is one of the main results of \cite{Steinbergandnora2020simplicityinversesemigroupetale}.
\begin{Theorem}[{\cite[Theorem 3.9]{Steinbergandnora2020simplicityinversesemigroupetale}}]\label{Thm: singular unique maximal containing tight} Let $S$ be a congruence-free inverse semigroup with zero and $K$ a field. Then $\SingularIdeal$
is the unique maximal ideal of $KS$ containing $\TightIdeal$. In particular, $\faktor{KS}{\TightIdeal}$ is simple
if and only if $\TightIdeal=\SingularIdeal$.
\end{Theorem}
We remark that the fact that the inverse semigroup is congruence-free is required for the uniqueness and simplicity hypotheses in the theorem above, as any congruence of the inverse semigroup would give rise to an ideal of the algebra. Without the assumption that $S$ is congruence-free, the fact that $\SingularIdeal$ contains $\TightIdeal$ remains. The natural partial order on $E_{(G, \Gamma)}$ applies to attain the following description of the elements of the singular ideal.
\begin{Proposition}\label{Prop: singular hands on description}
    Let $a\in \KS$. Then $a$ is singular if and only if for all $p\in \PathSet$, there exists some $q\in\range(p)\PathSet$, such that $ apq=0$.
\end{Proposition}
\begin{proof}
    Observe that for $p,q\in \PathSet$, we have $qq^*\leq pp^*$ if and only if $q\in p\PathSet$. Suppose $a$ is singular. Then for all $p\in \PathSet$, there exists some $q\in \range(p)\PathSet$ such that $a(pq)(pq)^*=0$. It quickly follows that $a(pq)(pq)^*=0$ if and only if $apq=0$.
\end{proof}
Item (3) of Proposition \ref{Prop: tight hands on descriptions} and Proposition \ref{Prop: singular hands on description} provide an alternative proof that $\TightIdeal\subseteq\SingularIdeal$. We are interested in determining when $\TightIdeal\neq\SingularIdeal$.

\section{Collapsing self-similar groupoids}\label{Sec: Collapsing Groupoids}
Fix a self-similar groupoid $(G, \Gamma)$. It turns out that if there are distinct $u,v\in\VertSet$ such that $G\cap\text{Iso}(u,v)\neq \emptyset$, then we may study a quotient of $G$ acting on a corresponding quotient of $\PathSet$ in order to determine whether $\TightIdeal$ and $\SingularIdeal$ coincide in the original algebra $\KS$. For the groupoid quotient, we collapse $G$ onto its categorical skeleton. Because we regard $G$ as a subgroupoid of $\textbf{Iso}(\PathSet)$, the choice of quotient of each orbit of $G$ determines the quotient of $\Gamma$. The two quotients are specified as follows.

For a vertex $v\in \VertSet$, the orbit of $G$ containing $\id_v$ is the subgroupoid,
\[G\id_v G = \{a\id_v b: a,b\in G\}.\]
Choose a set of orbit representatives $R\subseteq\VertSet$ such that,
 \[G=\bigsqcup_{v\in R}G\id_v G.\]

For each $v\in R$ and each $\id_u$ in $G\id_v G$ with $u\neq v$ choose some $f_u\in G$ such that $\id_u={f}^{-1}_u\id_v f_u$. Let $F_v=\{\id_v, \id_u, f_u, f_u^{-1}: \id_u\in G\id_v G\}$ and $F=\bigsqcup_{v\in R}F_v$. We have constructed $F$ to be a subgroupoid of $G$ with the same set of objects as $G$ and such that $F\cap\text{Iso}(u,v)$ is a singleton whenever $G\cap\text{Iso}(u,v)$ is nonempty. By choosing $R$ and $F$, we are choosing exactly how $G$ will quotient onto its categorical skeleton.

Now define the quotient map $\alpha: G\rightarrow \faktor{G}{\psi_F}$, where $\psi_F$ is the least congruence of $G$ which restricts to the universal congruence on $F_v$ for each $v\in R$, and let $L=\alpha(G)$. Indeed, the $\psi_F$-equivalence class of $F_v$ is exactly $F_v$ for each $v\in R$. Viewing $G$ in $\textbf{Iso}(\PathSet)$, we extend $\alpha$ to all of $\textbf{Iso}(\PathSet)$. Let us construct the corresponding quotient of $\Gamma$ determined by $R$ and $F$.

We say that paths $p, q\in \PathSet$ are $F$-equivalent, denoted by $p\sim_F q$, if there exists some $f\in F$ such that $f(p)=q$. It is easy to see that $p\sim_F q$ implies $\source(p)\sim_F \source(q)$. The equivalence relation $\sim_F$ induces a directed graph homomorphism $\beta$ which maps $\PathSet\rightarrow \faktor{\PathSet}{\sim_F}$. Observe that for a given $p\in \PathSet$, the paths in the $\sim_F$-equivalence class of $p$ are in bijection with the vertices in the $\sim_F$-equivalence class of $\source(p)$. Importantly, this implies that $\beta$ restricts to a bijection on each tree $v\PathSet$, that is, $\beta$ is locally bijective. The restriction of $\beta$ to $\VertSet\cup\EdgeSet$ induces a graph homomorphism on $\Gamma$ whose image we denote $\Lambda=\beta(\Gamma)$.

Because of the local bijectivity of $\beta$, the image $\beta(\PathSet)$ may be identified with $\Lambda^*$, though there is not a canonical choice of identification. It suffices to define this identification on $\beta(v\PathSet)$ for each $v\in R$. Because $\alpha$ quotients both $G$ and $\textbf{Iso}(\PathSet)$ by $F$, and $\beta$ quotients $\PathSet$ by the action of $F$, we now see that $\alpha(\textbf{Iso}(\PathSet))=\textbf{Iso}(\beta{\PathSet})$. Thus, by identifying $\beta{\PathSet}$ with $\Lambda^*$, we have identified $\alpha(\textbf{Iso}(\PathSet))$ with $\textbf{Iso}(\Lambda^*)$. Now $L=\alpha(G)$ sits inside $\textbf{Iso}(\Lambda^*)$, and thus comes equipped with an action on $\Lambda^*$.

\begin{Proposition}\label{Prop: L,Lambda is self-similar}
The action of $L$ on $\Lambda^*$ is a self-similar action.
\end{Proposition}
\begin{proof}
    Let $m\in L$ and let $e\in \Lambda^1$ be in the domain of $m$. There exist some $g\in \alpha^{-1}(m)$ and $d\in \beta^{-1}(e)$ such that,
    \[g(dq)=g(d)g|_d(q),\]
    for all $q\in\range(d)\PathSet$ because $(G, \Gamma)$ is self-similar. Observe that $\range(d)\PathSet$ is in bijection with $\range(e)\Lambda^*$. It is then clear that
    \[m(ep)=\alpha(g)(\beta(dq)) =\alpha(g)(\beta(d))\alpha(g|d)(p)\]
    for all $p$ in $\range(e)\Lambda^*$. Thus, $m|_{e}=\alpha(g|_{d})$ and the action of $L$ is self-similar.
\end{proof}

Let us call $(L, \Lambda)$ the collapsed self-similar groupoid of $(G, \Gamma)$. Note that by construction, $(L, \Lambda)$ is a \textit{group bundle}, that is, a groupoid which is a disjoint union of groups. Before we address the main motivation for introducing collapsed self-similar groupoids, let us consider what happens when $S_{(G, \Gamma)}$ is congruence-free.

\begin{Proposition}\label{Prop: collapsed L,Lambda is congruence-free and strongly connected}
    If the inverse semigroup $S_{(G, \Gamma)}$ is congruence-free, then the inverse semigroup $S_{(L, \Lambda)}$ is congruence-free and $\Lambda$ is strongly connected.
\end{Proposition}
\begin{proof}
Suppose that $S_{(G, \Gamma)}$ is congruence-free. By Corollary \ref{Cor: CongruenceFreeAppliedS_G with nonempty graph}, $(G, \Gamma)$ satisfies the following:
\begin{enumerate}
        \item  $|\source^{-1}(v)|> 1$ for all $v\in \VertSet$,
        \item  for each pair of distinct vertices $v,w \in\VertSet$, there exists $p\in w\PathSet$ and $q\in v\PathSet$ such that $G$ intersects $\text{Iso}(\range(p), v)$ and $\text{Iso}(\range(q), w)$.
    \end{enumerate}
    Observe that for each $v\in \VertSet$, $F\cap \text{Iso}(v,v) = \id_v$, and so for distinct $e,f\in v\EdgeSet$, we have $e\nsim_F f$. It follows that $\beta: \Gamma\rightarrow \Lambda$ is injective on $v\EdgeSet$, and so if $\Gamma$ satisfies condition 1., $\Lambda$ must satisfy condition 1.

    Now let $x,y\in \Lambda^0$ and suppose $v\in\beta^{-1}(x)$ and $w\in \beta^{-1}(y)$. By assumption, there exists some $p\in w\PathSet$ such that $G\cap \text{Iso}(\range(p),v)$ is nonempty. But then $\beta(p)$ is a path from $y$ to $\beta(\range(p))$. However $\range(p)\sim_F v$ and so $\beta(\range(p))=\beta(v)=x$. Thus, $\beta(p)$ is a path from $y$ to $x$. By a symmetric argument, there exists a path in $\Lambda$ from $x$ to $y$, and thus, $\lambda$ is strongly connected.

    It follows that $(L, \Lambda)$ trivially satisfies property 2.\ and so by Corollary \ref{Cor: CongruenceFreeAppliedS_G with nonempty graph}, $S_{(L, \Lambda)}$ is congruence-free.
\end{proof}
Now onto the primary concern: the singular and tight ideals of $KS_{(G, \Gamma)}$ and $KS_{(L, \Lambda)}$. Let $H=G\cap\text{Iso}(v,v)$ be the isotropy group acting on $v\PathSet$ for some vertex $v\in\VertSet$. Let $M$ be the image of $H$ under $\alpha$, that is, $M=L\cap \text{Iso}(\beta(v),\beta(v))$. The following lemma is really shown in Proposition \ref{Prop: L,Lambda is self-similar}.
\begin{Lemma}\label{Lem: collapsing morphism and restriction by paths}
    Let $h\in H$ and let $p\in v\Gamma^*$. Then $\alpha(h|_p) = \alpha(h)|_{\beta(p)}$.
\end{Lemma}
Recall that the actions of $L$ on $\Lambda^*$ and $G$ on $\PathSet$ extend linearly to actions of $KS_{(L, \Lambda)}$ on $K\Lambda^*$ and $KS_{(G, \Gamma)}$ on $K\Gamma^*$. Linearly extend the map $\alpha$ to map $KG$ to $KL$.
\begin{Proposition}
    Let $a\in KH$ and $p\in v\Gamma^*$. Then $ap=0$ if and only if $\alpha(a)\beta(p)=0$.
\end{Proposition}
\begin{proof}
    Let $g,h\in \Support a$. Notice that $\alpha$ restricts to a bijection between $K H$ and $KM$, and $\beta$ restricts to a bijection between $v \Gamma^*$ and $\beta(v)\Gamma^*$. Then $g(p)=h(p)$ if and only if $\alpha(g)(\beta(p))=\alpha(h)(\beta(p))$. By Lemma \ref{Lem: collapsing morphism and restriction by paths} we have $g|_p=h|_p$ if and only if $\alpha(g)|_{\beta(p)}=\alpha(h)|_{\beta(p)}$. It follows that $ap=0$ if and only if $\alpha(a)\beta(p) = 0$.
\end{proof}
Observe that $\beta: \Gamma^*\rightarrow \Lambda^*$ uniquely extends to a map $\beta: \Gamma^\omega\rightarrow\Lambda^\omega$. Then applying Propositions \ref{Prop: tight hands on descriptions} and \ref{Prop: singular hands on description}, we immediately have the following.
\begin{Corollary}\label{Cor: ess is tight in groupoid iff singular is tight in collapsed groupoid}
    The singular and tight ideals of $KS_{(L, \Lambda)}$ coincide if and only if the singular and tight ideals of $KS_{(L, \Lambda)}$ coincide.
\end{Corollary}
\begin{Corollary}\label{Cor: groupoid algebra simple iff collapsed groupoid algebra simple}
    The algebra $\faktor{KS_{(G, \Gamma)}}{\TightIdeal}$ is simple if and only if the algebra $\faktor{KS_{(L, \Lambda)}}{\TightIdeal}$ is simple.
\end{Corollary}
\begin{proof}
    This follows from Corollary \ref{Cor: ess is tight in groupoid iff singular is tight in collapsed groupoid} and Theorem \ref{Thm: singular unique maximal containing tight}.
\end{proof}
\section{Supporting \texorpdfstring{$\SingularNotTightIdeal$}{the singular ideal minus the tight ideal} on the groupoid algebra.}\label{Sec: Supporting essnottight on groupoid subalgebra}
We return to the analysis of the singular and tight ideals of $\KS$ for a self-similar groupoid $(G, \Gamma)$ of type (CF). The goal of this section is to show that if $\SingularNotTightIdeal$ is nonempty, it must intersect the groupoid algebra $KG$. Let $M=\PathSet G$ be a subsemigroup of $S_{(G, \Gamma)}$. We begin by showing that when nonempty, $\SingularNotTightIdeal$ intersects the subalgebra $KM$. For $a\notin\TightIdeal$, we say that an infinite path $\mathfrak{p}\in \Gamma^{\omega}$ \textit{witnesses} that $a\notin\TightIdeal$ if for all prefixes $p$ of $\mathfrak{p}$, we have $ap\neq 0$. Note that it follows from Proposition \ref{Prop: tight hands on descriptions} that whenever $a\notin\TightIdeal$, such a witness path must exist.
\begin{Lemma}\label{Lem: a witnessed by _w, aw witnessed by w*_w}
    If $\mathfrak{p}\in\Gamma^{\omega}$ witnesses that $a\notin \TightIdeal$, then $ap\notin \TightIdeal$ for all prefixes $p$ of $\mathfrak{p}$.
\end{Lemma}
\begin{proof}
    Suppose that $a\notin \TightIdeal$. Let $\mathfrak{p}\in\Gamma^{\omega}$ witness that $a\notin \TightIdeal$. Suppose $p,q\in \PathSet$ and $\mathfrak{q}\in\Gamma^{\omega}$ such that $\mathfrak{p}=pq\mathfrak{q}$. Then $apq \neq 0$ as $pq$ is a finite prefix of $\mathfrak{p}$. Thus, $q\mathfrak{q}$ witnesses that $ap\notin \TightIdeal$ for every finite prefix $p$ of $\mathfrak{p}$.
\end{proof}
\begin{Proposition}\label{Prop: EssNotTight intersects KM}
    If $\SingularNotTightIdeal$ is not empty, $\SingularNotTightIdeal$ contains an element of the subalgebra $KM$.
\end{Proposition}
\begin{proof}
    Suppose $a\in \SingularNotTightIdeal$ and let $\mathfrak{w}\in\Gamma^{\omega}$ witness that $a\notin \TightIdeal$. By Lemma \ref{Lem: a witnessed by _w, aw witnessed by w*_w}, $aw\notin \TightIdeal$ for any finite prefix $w$ of $\mathfrak{w}$. Additionally, $aw\in \SingularIdeal$ because $\SingularIdeal$ is an ideal.

    Let $a=\sum_{s\in S_{(G, \Gamma)}^{\sharp}}k_ss$. Each $s$ in the support of $a$ may be written in the canonical form $pgq^{*}$ with $p,q\in \PathSet$ and $g\in G$. Whenever a prefix $w$ of $\mathfrak{w}$ satisfies $|w|\geq |q|$, we have,
    \[pgq^{*} w = \begin{cases}
        0&\text{ if $q$ is not a prefix of $w$,}\\
        pgm &\text{ if $w=qm$ for some $m\in \PathSet$.}
    \end{cases}\]
    In either case, $pgq^{*} w$ is in $KM$. Therefore, we may choose a sufficiently long prefix $w$ of $\mathfrak{w}$ such that $sw\in KM$ for all $s\in \Support a$, in which case $aw\in KM$.
\end{proof}
\begin{Proposition}\label{Prop: aw = 0 iff a_pw = 0}
    Let $a\in KM$. We may uniquely write $a=\sum_{p\in \PathSet}pa_{p}$ where $a_p\in KG$ is nonzero for finitely many $p$. For any $w\in\PathSet$, $aw=0$ if and only if $a_pw=0$ for all $p\in \PathSet$.

    As a corollary, $a$ is singular if and only if $a_p$ is singular for all $p$.
\end{Proposition}
\begin{proof}
    Let $a = \sum_{p\in \PathSet}pa_{p}$. The uniqueness of this factorization follows from the canonical form of elements of $S_{(G, \Gamma)}$. Additionally, $\Support a_p \subseteq \bigsqcup_{v\in \VertSet}\text{Iso}(v,\range(p))$, that is, for all $g\in \Support a_p$, the image of $g$ is $\range(p)\PathSet$, and so, $pgw= pg(w)g|_w\neq 0$ whenever $w\in \dom g$. 
    
    For each $g\in G$ and $w\in \PathSet$, we have
    \[gw = \begin{cases}
        g(w)g|_w &\text{if } w\in \dom g,\\
        0 &\text{if } w\notin \dom g,
    \end{cases}\]
    where $|g(w)|=|w|$. If $|w|=n$, then whenever $a_pw\neq 0$, we may uniquely write $a_pw=\sum_{q\in \Gamma^n}qb_{pq}$ where $b_{pq}\in KG$ is nonzero for finitely many $q$. Supposing that $aw=0$, we have,
    \[0 = aw =\sum_{p\in \PathSet}pa_{p}w
    = \sum_{p\in \PathSet}\sum_{q\in \Gamma^n}pqb_{pq}\]
    with each $pq\neq 0$ as discussed. Whenever $pq=p'q'$ and $|q|=|q'|$, it is clear that $p=p'$ and $q=q'$. Thus, $aw=0$ implies that $b_{pq}=0$ for all $pq\in \PathSet$, which further implies that
    \[a_p w= \sum_{q\in \Gamma^n}qb_{pq} = 0.\]
    Conversely, suppose that $a_pw=0$ for all $p\in \PathSet$. It is immediate that
    \[aw = \sum_{p\in \PathSet}pa_{p}w = 0,\]
    thus, $aw=0$ if and only if $a_pw = 0$ for all $p$.

    Finally, it follows from Proposition \ref{Prop: singular hands on description} that if $a$ is singular, then $a_p$ is singular for all $p$. Conversely, if each $a_p$ is singular, then $a = \sum_{p\in \PathSet}pa_{p}$ is singular because $\SingularIdeal$ is an ideal.
\end{proof}
\begin{Corollary}\label{Cor: EssNotTight intersects KG}
    If $\SingularNotTightIdeal\neq \emptyset$, then $\SingularNotTightIdeal$ intersects the groupoid algebra $KG$.
\end{Corollary}
\begin{proof}
    By Proposition \ref{Prop: EssNotTight intersects KM}, if $\SingularNotTightIdeal\neq \emptyset$ there exists some $a\in KM\cap\SingularNotTightIdeal$. Uniquely express $a$ in the form $\sum_{p\in \PathSet}pa_{p}$ where $a_p\in KG$ as in Proposition \ref{Prop: aw = 0 iff a_pw = 0}. Then, by Proposition \ref{Prop: aw = 0 iff a_pw = 0}, each $a_p$ is singular. Similarly, since $a\notin \TightIdeal$, there is some $p\in\PathSet$ such that $a_p\notin\TightIdeal$.
\end{proof}
\section{Supporting \texorpdfstring{$\SingularNotTightIdeal$}{the singular ideal minus the tight ideal} when \texorpdfstring{$G$}{G} is contracting}\label{Sec: Supporting ideals in nucleus}
Section \ref{Sec: Supporting essnottight on groupoid subalgebra} establishes that when searching for elements of $\SingularNotTightIdeal$, we may restrict our search to the subalgebra $KG$. In the case of a contracting self-similar groupoid $(G, \Gamma)$ with nucleus $N$, we can further restrict the search to a finite subspace of the span of the nucleus $KN$. For a path $p\in \PathSet$, define the set
\[H_p=\{g\in G: g(p)=p, g|_p=g\}.\]
Notice that the condition that $g(p)=p$ implies that $\dom g= \im g = \source(p)\PathSet$ and the condition that $g|_p=g$ implies that $\source(p)=\range(p)$ as $\dom g|_p = \range(p)\PathSet$. Thus, for any path $q$ with $\source(q)\neq \range(q)$, $H_q$ is the empty set. Conversely, for a path $p$ with $\source(p)=\range(p)$, $H_p$ is not empty as $\id_{\source(p)}\in H_p$.
\begin{Proposition}\label{Prop: H_p subgroup in N}
    Whenever $H_p$ is nonempty, $H_p$ is a subgroup of $G$ contained in the nucleus $N$.
\end{Proposition}
\begin{proof}
    Suppose $H_p$ is nonempty. By the discussion above, $H_p$ is a subset of the isotropy group of $G$ acting on $\source(p)\PathSet$ and contains the identity of that isotropy group. Whenever $g,h\in H_p$, we have $g^*h(p)=p$ and
    \[(g^*h)|_p = g^*|_{h(p)}h|_p = (g|_{g^*h(p)})^*h = (g|_p)^*h = g^*h.\]
    Thus, $g^*h\in H_p$, making $H_p$ a subgroup of $G$. It is easy to see that $H_p\subseteq N$, because the condition $g|_p=g$ implies that $g|_{p^n}=g$ for all natural numbers $n$, and so $g\in N$.
\end{proof}
In light of Proposition \ref{Prop: H_p subgroup in N} and the condition that $g|_p=g$, we call $H_p$ the \textit{recurring subgroup} corresponding to $p$ whenever $H_p$ is nonempty. For a nonempty recurring subgroup $H_p$, we denote the span of $H_p$ in $\KS$ by $KH_p$. The span of the nucleus contains $KH_p$ by Proposition \ref{Prop: H_p subgroup in N}, and therefore $KH_p$ is a finite dimensional subspace of $\KS$. Observe that if $0\neq a\in KH_p$, then $(p^n)^*ap^n=(p^n)^*p^na = \source(p)a=a$ for all natural numbers $n$, and therefore $ap^n\neq 0$. It follows from Proposition \ref{Prop: tight hands on descriptions} that $a\notin\TightIdeal$, and so $KH_p$ only intersects $\TightIdeal$ at $\{0\}$.
\begin{Theorem}\label{Thm: EssNotTight intersects H_p}
    Let $(G, \Gamma)$ be a contracting self-similar groupoid with nucleus $N$. If $\SingularNotTightIdeal$ is nonempty then there exists some $p\in \PathSet$ such that $H_p$ is nonempty and $KH_p$ contains an element of $\SingularNotTightIdeal$.
\end{Theorem}
\begin{proof}
    First suppose that $\SingularNotTightIdeal$ is nonempty. Let $a=\sum_{g\in G}k_g g$ be in the intersection of $\SingularNotTightIdeal$ and $KG$ and suppose $a$ has minimal support. Let $\mathfrak{w}\in\Gamma^{\omega}$ witness that $a\notin\TightIdeal$ and let $w_n$ denote the length $n$ prefix of $\mathfrak{w}$ for any natural number $n$. We have
    \[0\neq aw_n = \sum_{g\in G}k_g g w_n =\sum_{g\in G}k_g g(w_n)g|_{w_n}.\]
    By Lemma \ref{Lem: a witnessed by _w, aw witnessed by w*_w}, $aw_n\notin\TightIdeal$ and because $\SingularIdeal$ is an ideal, $aw_n\in\SingularNotTightIdeal$. It follows from the minimality of $|\Support a|$ that $|\Support a| = |\Support aw_n|$, that is, $w_n\in \dom g$ and $g|_{w_n}\neq h|_{w_n}$ for all distinct $g$ and $h$ in the support of $a$. We may uniquely express $aw_n$ in the form
    \[aw_n=\sum_{q\in \Gamma^n}qb_q,\]
    with $b_q\in KG$ nonzero for finitely many $q$. By Proposition \ref{Prop: aw = 0 iff a_pw = 0} and Corollary \ref{Cor: EssNotTight intersects KG}, there exists some $q\in \Gamma^n$ such that $b_q\in \SingularNotTightIdeal$. Then by the minimality of $|\Support a|$, there is only one path $q$ such that $b_q\neq 0$, and therefore
    \[aw_n = \sum_{g\in \Support a}gw_n = \sum_{g\in \Support a}qg|_{w_n},\]
    or equivalently, $g(w_n)=q = h(w_n)$ for all $g,h\in \Support a$, and the sum, \[\sum_{g\in \Support a}g|_{w_n},\] is in $\SingularNotTightIdeal$.

    The above observations hold regardless of the natural number $n$. By the definition of contracting, there is some natural number $k$ such that whenever $n$ exceeds $k$, $g|_{w_n}$ is in the nucleus $N$ for all $g\in \Support a$. Then by the finiteness of $N$, there exists natural numbers $m<n$, each greater than $k$, such that $g|_{w_m}=g|_{w_n}$ for all $g\in \Support a$. Let $p\in \Gamma^{n-m}$ such that $w_mp = w_n$, and define \[b=\sum_{g\in \Support a}g|_{w_m},\]
    which lies in the intersection of $\SingularNotTightIdeal$ and $KN$ by the discussion above. Then, $g|_{p}=g$ and $g(p)=h(p)$ for all $h,g\in \Support b$. 
    
    Fix some $h\in \Support b$. We claim that $h^* b$ is an element of $\SingularNotTightIdeal$ supported on $H_p$. Indeed $h^* g(p) = p$ for all $g\in \Support b$. Additionally,
    \[(h^*g)|_p = (h^*)|_{g(p)}g|_p = (h^*)|_{h(p)}g|_p=(h|_{p})^*g|_p =h^*g,\]
    establishing that $h^*b\in KH_p$. Finally, $h^*b\in\SingularNotTightIdeal$ because $b\in\SingularIdeal$ and $KH_p$ does not intersect $\TightIdeal$ non-trivially.
\end{proof}
 In the following section, the recurring subgroups play a large role in determining the simplicity of the algebra $\faktor{\KS}{\TightIdeal}$ for a contracting self-similar groupoid $(G, \Gamma)$. Because the nucleus is finite, there are finitely many recurring subgroups. In the self-similar group case, an algorithm to determine all recurring subgroups is provided in Lemma 6.3 of \cite{gardella2025simplicitycalgebrascontractingselfsimilar} which, if needed, can be lifted to the self-similar groupoid setting with minor adjustments.
\section{A method to determine simplicity}\label{Sec: simplicity algorithm}
Let $(G, \Gamma)$ be a contracting self-similar groupoid with nucleus $N$. The \textit{Moore diagram} of $N$ is the labeled directed graph with vertex set $N$ and, for each $g\in N$ and $e\in \EdgeSet \cap\dom g$, an edge from $g$ to $g|_e$ labeled by $(e,g(e))$. Let $\HH$ be the subgraph of the Moore diagram of $N$ obtained by removing all edges except those with label $(e,e)$ for some $e\in \EdgeSet$, and simplify the edge labeling by replacing $(e,e)$ with $e$. Denote the subset of $N$ consisting of vertices that lie on a nonempty cycle in $\HH$ by $\CC$.

From $\CC$, we build a directed graph denoted by $\Delta$ which is used to determine whether $\SingularIdeal = \TightIdeal$. Regarding $\PathSet$ as a category with path concatenation as the partial operation, let us define a category action of $\PathSet$ on the subsets of $\CC$.

Let the domain of the action of a path $q\in \PathSet$ be all subsets of $\CC$ whose elements lie in $\text{Iso}(\range(q),\range(q))$. On such a subset $Y$, define the action,
\[q\cdot Y = \{g\in \CC: g(q)=q, g|_q\in Y\}.\]

Observe that for $pq \in \PathSet$ and $Y\subseteq \CC \cap \text{Iso}(\range(q),\range(q))$,
\begin{align*}
    p\cdot (q\cdot Y) &= \{g\in \CC: g(p)=p, g|_p\in q\cdot Y\} \\
    &= \{g\in \CC: g(p)=p, g|_p(q)=q, (g|_p)|_q\in Y\} \\
    &= \{g\in \CC: g(pq)=pq, g|_{pq}\in Y\} \\
    &= (pq) \cdot Y,
\end{align*}
and therefore this is a well-defined category action of $\PathSet$.

Notice that for $p\in \PathSet$, we have $p\cdot \{\id_{\range(p)}\} = \{g\in \CC: g(p)=p, g|_p = \id_\range(p)\}$ which contains $\{\id_{\source(p)}\}$. We simplify notation by writing $I_{p}=p\cdot\{\id_{\range(p)}\}$ for $p\in \PathSet$, and in the case of a vertex $v\in \VertSet$, this notation yields $I_v = v\cdot \{\id_v\} = \{\id_v\}$. Let $Q$ be the orbit of $\left\{I_v: v\in \VertSet\right\}$ under the action of $\PathSet$. Explicitly we have
\[Q = \{I_q: q\in\PathSet\}.\]
Let $\Delta$ be the \textit{Schreier graph} of the left action of $\EdgeSet$ on $Q$, that is, $Q$ is the vertex set of $\Delta$, and there is an edge labeled by $e\in \EdgeSet$ from $I_q$ to $e\cdot I_{q} = I_{eq}$ for each $I_q\in Q$ in the domain of the action of $e$. It follows that for any $p\in \PathSet$ such that $p\cdot I_q$ is defined, there is a path in $\Delta$ labeled by the reversal  $\rho(p)$ of the edge sequence of $p$.

Now suppose that there exists some nonempty recurring subgroup $H_p=\{g\in N: g(p)= p, g|_p = g\}$ for some $p\in \PathSet$. The following results show how we can use the graph $\Delta$ to determine if $KH_p$ intersects the singular ideal of $KS_{(G, \Gamma)}$. It follows from Proposition \ref{Prop: H_p subgroup in N} and the definition of $H_p$ that $H_p\subseteq \CC$.

A path $w\in \Gamma^*$ is said to be \textit{synchronizing} in $\Delta$ if every path labeled by $\rho(w)$ in $\Delta$ leads to the same vertex, that is, if $w\cdot I_q = w\cdot I_{q'}$ for all $q,q'\in \range(w)\Gamma^*$. Notice that when this is the case, every path labeled by $\rho(w)$ in $\Delta$ leads to $I_w = w\cdot I_{\range(w)}$. For this reason, we say a vertex $I$ of $\Delta$ is \textit{synchronized} by $w$ if $I=I_w$ for some synchronizing path $w$.

\begin{Proposition}\label{Prop: H_p cap I_q is a subgroup of H_p}
    Let $H_p$ be a recurring subgroup and let $I_q$ be a vertex of $\Delta$. Whenever $H_p\cap I_q$ is nonempty, it is a subgroup of $H_p$. Furthermore, $H_p\cap I_q$ is nonempty if and only if $\source(p)=\source(q)$.
\end{Proposition}
\begin{proof}
    First, suppose that $g,h\in H_p\cap I_q$. As $H_p$ is a group, $gh^* \in H_p$. Since $g,h\in I_q$, we have $g(q)=q=h(q)$ and therefore, $gh^*(q)=q$. Furthermore, $g|_{q}=\id_{\range(q)} = h|_{q}$. Thus, $h^*|_{q} = (h|_{h^*(q)})^* = (h|_q)^* = \id_{\range(q)}$. It is then easy to see that
    \[(gh^*)|_q = g|_{h^*(q)}h^*|_q = g|_{q}\id_{\range(q)}=\id_{\range(q)}.\]
    Therefore whenever $H_p\cap I_q$ is nonempty, it is a subgroup of $H_p$.

    Finally, observe that $\id_{\source(q)}$ is the only idempotent of $I_q$ and $\id_{\source(p)}$ is the only idempotent of $H_p$. Since the intersection $H_p\cap I_q$ contains an idempotent whenever it is nonempty, $H_p\cap I_q$ is nonempty if and only if $\source(p)=\source(q)$.
\end{proof}
In light of Proposition \ref{Prop: H_p cap I_q is a subgroup of H_p}, we introduce the map $\pi_{p,q}: KH_p\rightarrow K[\faktor{H_p}{H_p\cap I_q}]$ whenever $\source(p)=\source(q)$, induced by the canonical map from $H_p$ onto the set of left cosets of the subgroup $H_p\cap I_q$.
\begin{Proposition}\label{Prop: pi_pq(a)=0 iff aq=0}
    Let $H_p$ be a recurring subgroup and $I_q$ be a vertex of $\Delta$ such that $\source(p)=\source(q)$. Let $a\in KH_p$. Then $\pi_{p,q}(a)=0$ if and only if $aq=0$.
\end{Proposition}
\begin{proof}
    First, let $g,h\in H_p$. Then $\pi_{p,q}(g)=\pi_{p,q}(h)$ if and only if $g(H_p\cap I_q)=h(H_p\cap I_q)$. This occurs if and only if $h^{-1}g \in H_p\cap I_q$, which is equivalent to $h^{-1}g(q)=q$ and $(h^{-1}g)|_{q}=\id|_{\range(q)}$. Notice that $h^{-1}g(q)=q$ if and only if $g(q)=h(q)$ and therefore
    \[(h^{-1}g)|_{q} = h^{-1}|_{g(q)}g|_{q}=h^{-1}|_{h(q)}g|_{q}=(h|_{q})^{-1}g|_{q}.\]
    It follows that $g(H_p\cap I_q)=h(H_p\cap I_q)$ if and only if $g(q)=h(q)$ and  $g|_q=h|_q$.
    By Definition \ref{Def: Associated Inv SG}, $g(q)=h(q)$ and $g|_q=h|_q$ is equivalent to the condition that $gq=hq$, when $g,h$, and $q$ are regarded as elements of $KS_{(G, \Gamma)}$. Let $a=\sum_{g\in H_p}k_gg$ be an element of $KH_p$. It follows that $\pi_{p,q}(a) = 0$ if and only if for all $g\in \Support a$, we have that
    \begin{align*}
        \sum_{\pi_{p,q}(h)=\pi_{p,q}(g)} k_h = 0 &\iff \sum_{hq=gq} k_h=0\\
        &\iff \sum_{hq=gq} k_hgq=0\\
        &\iff \sum_{hq=gq} k_hhq=0,\\
    \end{align*}
    and therefore $\pi_{p,q}(a) = 0$ if and only if $aq=0$.
\end{proof}
\begin{Lemma}\label{Lem: if Iq is synchronized then Iwq is synchronized}
    Let $I_q$ be synchronized by $q$ in $\Delta$. Then $w\cdot I_q$ is synchronized by $w$ in $\Delta$ for every $w\in \Gamma^{*}\source(q)$.
\end{Lemma}
\begin{proof}
    Let $w\in \Gamma^{*}\source(q)$. Note that $w\cdot I_q = I_{wq}$. Observe that every path labeled by $\rho(wq)$ in $\Delta$ begins with $\rho(q)$. By assumption, every path in $\Delta$ labeled by $\rho(q)$ ends at $I_q$, and since there is at exactly one path leaving $I_q$ labeled by $\rho(w)$, it must be the case that every path labeled $\rho(wq)$ ends at $I_{wq}$. Therefore, $I_{wq}$ is synchronized by $wq$ in $\Delta$.
\end{proof}
\begin{Lemma}\label{Lem: for all I_p there is a synchronized I_pq}
    Let $I_q$ be a vertex of $\Delta$. Then there exists some $w\in \range(q)\PathSet$ such that $I_{qw}$ is synchronized by $qw$ in $\Delta$.
\end{Lemma}
\begin{proof}
    Define $V_q=\{I_{qp}: p\in\range(q)\PathSet\}$. This is a finite set as it is a subset of the power set of $\CC$. The set $V_q$ is partially ordered by $\subseteq$, and it is clear to see that $I_{a}\subseteq I_{ab}$ for all $a, ab\in q\PathSet$. As a finite poset, $V_q$ must have a maximal element, necessarily of the form $I_{qw}$ for some $w\in \range(q)\PathSet$. Because it is maximal, $I_{qw}=I_{qwz}$ for all $z\in \range(w)\PathSet$. Then every path labeled by $\rho(qw)$ in $\Delta$ leads to $I_{qw}$. Thus, $I_{qw}$ is synchronized by $qw$ in $\Delta$.
\end{proof}
\begin{Proposition}\label{Prop: a in essIdeal iff a satisfies synchronized verts}
    Let $H_p$ be a recurring subgroup and let $a\in KH_p$. Then $a\in \SingularIdeal$ if and only if $\pi_{p,q}(a)=0$ for every synchronized vertex $I_q$ where $\source(q)=\source(p)$.
\end{Proposition}
\begin{proof}
    Let $a\in KH_p$ and let $\source(p)=v$. First, suppose that $a\in \SingularIdeal$. Let $I_q$ be synchronized by $q$ with $\source(q)=\source(p)$. Since $a\in \SingularIdeal$, there exists some $q'\in \range(q)\PathSet$ such that $aqq'=0$. By Proposition \ref{Prop: pi_pq(a)=0 iff aq=0}, $\pi_{p,qq'}(a)=0$. But because $I_q$ is synchronized, $I_q = I_{qq'}$, and therefore $\pi_{p,q}(a)=\pi_{p,qq'}(a)=0$.

    Now suppose that for every synchronizing path $q$ with $\source(q)=\source(p)$, we have $\pi_{p,q}(a)=0$. Let $w\in v\PathSet$. In order to show that $a\in \SingularIdeal$, we must find some $w'\in \range(w)\PathSet$ such that $aww'=0$. By Lemma \ref{Lem: for all I_p there is a synchronized I_pq}, there exists some $w'$, necessarily in $\range(w)\PathSet$, such that $I_{ww'}$ is synchronized by $ww'$. Notice that $\source(ww')=\source(p)=v$, and therefore $\pi_{p,ww'}(a)=0$. By Proposition \ref{Prop: pi_pq(a)=0 iff aq=0} $aww'=0$, and so $a\in\SingularIdeal$.
\end{proof}
\begin{Theorem}\label{Thm: Ess neq tight iff H_p and minimal vertices of synchronized verts of delta}
    Let $(G, \Gamma)$ be a self-similar groupoid of type (CF). Then $\SingularIdeal\subsetneq\TightIdeal$ if and only if there exists some recurring subgroup $H_p$ and some element $0\neq a\in KH_p$ such that $\pi_{p,q}(a)=0$ for every synchronized vertex $I_q$ where $\source(q)=\source(p)$.
\end{Theorem}
\begin{proof}
    First, suppose that $\SingularIdeal\subsetneq\TightIdeal$. Then by Theorem \ref{Thm: EssNotTight intersects H_p}, there exists some recurring subgroup $H_p\subseteq G$ such that $KH_p$ intersects $\SingularNotTightIdeal$. It follows from Proposition \ref{Prop: a in essIdeal iff a satisfies synchronized verts} that there is some $0\neq a\in KH_p$ such that $\pi_{p,q}(a)=0$ for every synchronized vertex $I_q$ where $\source(q)=\source(p)$.

    Alternatively, if there exists some $0\neq a\in KH_p$ such that $\pi_{p,q}(a)=0$ for every synchronized vertex $I_q$ where $\source(q)=\source(p)$, then $a\in \SingularIdeal$ by Proposition \ref{Prop: a in essIdeal iff a satisfies synchronized verts}. Recall that $KH_p\cap \TightIdeal=\{0\}$, and so $a\in \SingularNotTightIdeal$. Thus, $\SingularIdeal\subsetneq\TightIdeal$.
\end{proof}
\begin{Corollary}\label{Cor: algebra quotient tight simple iff ker of recurring subgroup maps is 0}
    For a self-similar groupoid $(G, \Gamma)$ of type (CF), $\faktor{K_0S_{(G, \Gamma)}}{\TightIdeal}$ is simple if and only if for every recurring subgroup $H_p$,
    \[\bigcap\ker\pi_{p,q}=\{0\},\]
     where the intersection is taken over every synchronized vertex $I_q$ where $\source(q)=\source(p)$.
\end{Corollary}
\begin{proof}
    This is an immediate consequence of Theorem \ref{Thm: Ess neq tight iff H_p and minimal vertices of synchronized verts of delta} and Theorem \ref{Thm: singular unique maximal containing tight}.
\end{proof}
It follows from Corollary \ref{Cor: ess is tight in groupoid iff singular is tight in collapsed groupoid} that in order to determine whether the singular and tight ideals coincide in the contracted algebra of the inverse semigroup associated to a congruence-free self-similar groupoid $(G, \Gamma)$, we may first collapse the self-similar groupoid using the pair of maps $(\alpha,\beta):(G, \Gamma)\rightarrow (L, \Lambda)$ as described in Section \ref{Sec: tight and singular ideals}, and then apply the algorithm described above to $(L, \Lambda)$. This substituion has the benefit that the graph $\Lambda$ is strongly connected by Proposition \ref{Prop: collapsed L,Lambda is congruence-free and strongly connected}. In this case, we can improve the utility of Proposition \ref{Prop: a in essIdeal iff a satisfies synchronized verts} by the following observation.

\begin{Proposition}\label{Prop: synchronized iff minimal when strongly connected}
    Let $(L, \Lambda)$ be a self-similar groupoid of type (CF) and suppose $\Lambda$ is strongly connected. Let $\Delta$ be the graph associated to $(L, \Lambda)$ as in Proposition \ref{Prop: a in essIdeal iff a satisfies synchronized verts}. Then a vertex $I$ of $\Delta$ is synchronized if and only if it is in a minimal component of $\Delta$ with respect to reachability. Moreover, $\Delta$ has a unique minimal component.
\end{Proposition}
\begin{proof}
    Let $I_p$ and $I_q$ be two vertices of $\Delta$ and suppose that $I_p$ is synchronized by $p$ in $\Delta$. Because $\Lambda$ is strongly connected, there exists a path $w\in \Lambda^*$ from $\range(p)$ to $\source(q)$. Then the reverse $\rho(w)$ labels a path from $I_q$ to $I_{wq}$ in $\Delta$. Now, $\source(w)=\range(p)$, so the reverse $\rho(p)$ labels a path from $I_{wq}$ to $I_{pwq}$ in $\Delta$. Since $p$ is synchronizing, we must have that $I_{pwq}=I_p$, and therefore $I_p$ is reachable from every other vertex of $\Delta$. It follows that $I_p$ is in what must be the unique minimal component of $\Delta$ with respect to reachability.

    Now suppose that $I_p$ is in a minimal component of $\Delta$ with respect to reachability. By Lemma \ref{Lem: for all I_p there is a synchronized I_pq}, there exists a synchronizing path of the form $pq$, such that $I_{pq}$ is a synchronized vertex. By the discussion above, $I_{pq}$ is reachable from $I_p$.Now since $I_{p}$ is in a minimal component of $\Delta$ with respect to reachability, $I_{pq}$ must be in that same minimal component, so there exists some path in $\Delta$, labeled by $\rho(w)$ say, from $I_{pq}$ back to $I_p$. Then $I_p = w\cdot I_{pq}=I_{wpq}$ which is synchronized by Lemma \ref{Lem: if Iq is synchronized then Iwq is synchronized}. This concludes the proof.
\end{proof}
The convenience of Proposition \ref{Prop: synchronized iff minimal when strongly connected} lies in the fact that algorithms for finding strongly connected components of a directed graph may be used in place of algorithms for finding synchronizing paths.
\section{Simplicity of ample groupoid \texorpdfstring{$C^*$}{C*}-algebras}\label{Sec: simplicity of cstar algebras}
As was mentioned in the introduction, the work of Exel and Steinberg allows us to use inverse semigroup algebras to study \'etale groupoid algebras. It is shown in \cite[Corollary 2.14]{Steinbergandnora2020simplicityinversesemigroupetale} that given a field $K$ and an inverse semigroup $S$ with zero, the quotient of the inverse semigroup algebra by the tight ideal, $\faktor{KS}{\TightIdeal}$, is isomorphic to the Steinberg algebra $K\mathcal{G}_{T}(S)$, where $\mathcal{G}_{T}(S)$ is the tight groupoid of the inverse semigroup in the sense of \cite{exel2008inversesemigroupscombinatorialcalgebras}. When $\mathcal{G}_{T}(S)$ is ample, we can leverage the results \cite{STEINBERG2010689} which shows that the Steinberg algebra $\mathbb{C}\mathcal{G}_{T}(S)$ embeds densely into what is called the reduced $C^*$-algebra of $\mathcal{G}_{T}(S)$. Finally, under suitable conditions on the groupoid, the simplicity of the complex Steinberg algebra and the simplicity of the reduced $C^*$-algebra of an ample groupoid were shown to coincide in \cite{brix2025hausdorffcoversnonhausdorffgroupoids}.

The tight groupoid of inverse semigroups associated to self-similar groupoids is ample, and so the connections mentioned above provide a natural application of our results on $\KS$ to the reduced $C^*$-algebra of the tight groupoid of $S_{(G, \Gamma)}$. Thus far we have chosen to operate in the language of inverse semigroups, however, the connections to $C^*$-algebras require that we use the language of \'etale groupoids, which we introduce now. 

A \textit{topological groupoid} is a groupoid equipped with a topology, such that the multiplication, inversion, source, and range maps are each continuous. An \textit{\'etale groupoid} is a topological groupoid where the unit space is locally compact and Hausdorff and the range map is a local homeomorphism. An \textit{ample} groupoid is an \'etale groupoid with a totally disconnected unit space. A local bisection of a groupoid $\mathcal{G}$ is a subset $B\subseteq \mathcal{G}$ such that the source and range maps restrict to homeomorphisms $\source|_B:B\rightarrow \source(B)$, $\range|_B:B\rightarrow \range(B)$. If $\source(B)=\mathcal{G}^0$, $B$ is a bisection. When $\mathcal{G}$ is \'etale, any open bisection $B\subseteq \mathcal{G}$ is locally compact and Hausdorff. See \cite{exel2008inversesemigroupscombinatorialcalgebras,Paterson1998GroupoidsIS,renaultcstaralgebra, STEINBERG2010689} for more on \'etale and ample groupoids.

Given an ample groupoid $\mathcal{G}$ and a ring $K$, an important algebra associated to $\mathcal{G}$ is the Steinberg algebra $K\mathcal{G}$. The Steinberg algebra $K\mathcal{G}$ is the $K$-span of the characteristic functions $1_U: \mathcal{G}\rightarrow K$ for each compact local bisection $U\subseteq \mathcal{G}$. The theory of Steinberg algebras was introduced and developed in \cite{STEINBERG2010689}.

The \textit{singular ideal} (also called the singular ideal \cite{brix2025hausdorffcoversnonhausdorffgroupoids, Steinbergandnora2020simplicityinversesemigroupetale}), of the Steinberg algebra $K\mathcal{G}$ is defined as
\[\mathcal{J}_{K}=\{f\in K\mathcal{G}: \text{supp}(f)\text{ has empty interior}\},\]
where $\text{supp}(f)$ is the subset $\mathcal{G}$ which is not sent to $0$.
The connection between singular ideals of Steinberg algebras and singular ideals of contracted inverse semigroup algebras is as follows. Associated to any inverse semigroup $S$ is the \textit{tight groupoid} \cite{exel2008inversesemigroupscombinatorialcalgebras} $\mathcal{G}_{T}(S)$ which turns out to be ample. It is shown in \cite[Corollary 2.14]{Steinbergandnora2020simplicityinversesemigroupetale} that the Steinberg algebra $K\mathcal{G}_{T}(S)$ of the tight groupoid of the inverse semigroup with zero $S$ is isomorphic to the quotient of the contracted inverse semigroup algebra by the tight ideal, $\faktor{K_0S}{\TightIdeal}$. The singular ideal $\mathcal{J}_K$ of the Steinberg algebra is identified with the image of the singular ideal $\SingularIdeal$ under the quotient map, hence the coordinating nomenclature.

Let $S=S_{(G, \Gamma)}$ be the inverse semigroup associated to some self-similar groupoid $(G, \Gamma)$ as in Definition \ref{Def: Associated Inv SG}. The tight groupoid of $S$ may be recovered from the action of $G$ on $\PathSet$ as follows. Recall the action of $S$ on $\PathSet$ given in Section \ref{Sec: Preliminaries}. This action naturally extends to an action of $S$ on $\Gamma^\omega$ by partial homeomorphisms. We equip $\Gamma^\omega$ with a topology generated by the basis of open sets of the form $q\Gamma^\omega$ where $q\in \PathSet$. Let $\mathcal{G}(S, \Gamma^{\omega})$ be the groupoid of germs of this action, as constructed in \cite{exel2008inversesemigroupscombinatorialcalgebras}. Explicitly, the groupoid of germs is
\[\mathcal{G}(S, \Gamma^{\omega})=\faktor{\{(s,\mathfrak{p})\in S^\sharp\times \Gamma^\omega: s=agb^*, \mathfrak{p}\in b\Gamma^\omega\}}{\asymp},\]
where $(s,\mathfrak{p})\asymp(t,\mathfrak{q})$ if and only if $\mathfrak{p}=\mathfrak{q}$ and there exists some prefix $w$ of $\mathfrak{p}$ such that $s$ and $t$ agree on $w\Gamma^\omega$. The $\asymp$-equivalence class of $(s,\mathfrak{p})$ is denoted by $[s,\mathfrak{p}]$ and is also called the germ of $(s,\mathfrak{p})$. The relation $\asymp$ is known as the germ relation. The multiplication of $\mathcal{G}$ is given by $[s,t(\mathfrak{p})][t,\mathfrak{p}]=[st,\mathfrak{p}]$, the inverse given by $[s,\mathfrak{p}]^{-1}=[s^*,s(\mathfrak{p})]$, and the source map given by
$\source([s,\mathfrak{p}])=[s^*s,\mathfrak{p}]$.

Recall that the nonzero idempotents of $S$ each take the form $aa^*$ where $a\in \PathSet$. If $\mathfrak{p}\in a\Gamma^\omega$, it is easy to see that $(aa^*,\mathfrak{p})\asymp(bb^*,\mathfrak{p})$ for any prefix $b$ of $\mathfrak{p}$, and therefore we may identify the unit space of $\mathcal{G}(S, \Gamma^\omega)^0$ with $\Gamma^\omega$. We then equip $\mathcal{G}(S, \Gamma^\omega)$ with the topology generated by the sets of the form
\[[s,U]=\{[s,\mathfrak{p}], \mathfrak{p}\in U\subseteq \Gamma^\omega\}\]
where $U$ is an open subset of the domain of $s$. It is shown in \cite[Proposition 4.14]{exel2008inversesemigroupscombinatorialcalgebras} that this topology makes $\mathcal{G}(S,\Gamma^\omega)$ a topological groupoid. Furthermore, $\Gamma^\omega$ is locally compact and Hausdorff so it follows from \cite[Proposition 4.17]{exel2008inversesemigroupscombinatorialcalgebras} that $\mathcal{G}(S,\Gamma^\omega)$ is an \'etale groupoid. The groupoid of germs $\mathcal{G}(S,\Gamma^\omega)$ is in fact an ample groupoid, as $\Gamma^\omega$ is totally disconnected.

It turns out that the action of $S$ on $\Gamma^\omega$ can be identified with the standard action of $S$ on the space of tight characters of $S$. It follows that the groupoid of germs $\mathcal{G}(S,\Gamma^\omega)$ coincides with the tight groupoid $\mathcal{G}_T(S)$; see \cite{SteinbergandNora2023} for this identification in the self-similar group case.

Another important algebra associated to $\mathcal{G}$ is the reduced $C^*$-algebra, denoted by $C_r^*(\mathcal{G})$. It is shown in \cite{clark2022Steinbergalgebraapproachetale} that the reduced $C^*$-algebra $C_r^*(\mathcal{G})$ is the closure of the Steinberg algebra $\mathbb{C}(\mathcal{G})$ with respect to what is known as the reduced norm.

It is possible to view elements of $C_r^*(\mathcal{G})$ as maps from $\mathcal{G}$ to $\mathbb{C}$ via Renault's $j$-map \cite{renaultcstaralgebra}, see also \cite[Section 4.1]{clark2019simplicityalgebrasassociatednonhausdorff}, which we do in the following definition. The $C^*$-algebraic \textit{singular ideal} of $C_r^*(\mathcal{G})$ is defined as
\[\mathcal{J}=\{f\in C_r^*(\mathcal{G}): \text{supp}(j(f))\text{ has empty interior}\},\]
where $\text{supp}(j(f))$ is the subset of the domain of $j(f)$ which is not sent to $0$.

Given an ample groupoid $\mathcal{G}$, it is unknown in general whether the singular ideal of $\mathbb{C}\mathcal{G}$ is dense in the singular ideal of $C^*_r(\mathcal(G))$, or whether the simplicity of $\mathbb{C}\mathcal{G}$ and $C^*_r(\mathcal{G})$ coincide, however progress is made in \cite{brix2025hausdorffcoversnonhausdorffgroupoids} for ample groupoids satisfying a certain finiteness condition. In particular, simplicity of the two algebras is shown to coincide for the ample groupoids coming from contracting self-similar groups. Inspired by the results of \cite[Section 7.3]{brix2025hausdorffcoversnonhausdorffgroupoids}, we show that simplicity of the Steinberg algebra and the reduced $C^*$-algebra associated to a contracting self-similar groupoid coincide.

Let $(G, \Gamma)$ be a contracting self-similar groupoid with nucleus $N$. Let $\mathcal{G}=\mathcal{G}_{(G, \Gamma)}$ be the ample groupoid associated to $(G, \Gamma)$, that is, the tight groupoid of $S=S_{(G, \Gamma)}$. Borrowing notation from \cite{brix2025hausdorffcoversnonhausdorffgroupoids}, let $X$ be the unit space of $\mathcal{G}$, identified with $\Gamma^\omega$. Finally, let $\overline X$ be the closure of $X$ in $\mathcal{G}$ and let $[s,wX]$ denote $\{[s,\mathfrak{p}]\in \mathcal{G}: \mathfrak{p}\in wX\}$. 

For each $\mathfrak{p}\in X$, let $\mathcal{G}^{\mathfrak{p}}_{\mathfrak{p}}$ be the isotropy group of $\mathcal{G}$ with source and range $\mathfrak{p}$, and let $\overline{X}(\mathfrak{p}) = \overline{X}\cap \mathcal{G}^{\mathfrak{p}}_{\mathfrak{p}}$.
\begin{Proposition}\label{Prop: finite closure of X(w)}
    The set $\overline{X}(\mathfrak{p})$ is finite because it is contained in $\{[pnp^*, \mathfrak{p}]: n\in N, p\text{ is a prefix of }\mathfrak{p}\}\cap \mathcal{G}_{\mathfrak{p}}^{\mathfrak{p}}$.
\end{Proposition}
\begin{proof}
Suppose that $[s, \mathfrak{q}]\in\overline{X}(\mathfrak{p})$. Then $[s, \mathfrak{q}]\in \mathcal{G}_{\mathfrak{p}}^{\mathfrak{p}}$ so $\mathfrak{q}=\mathfrak{p}$. Expressed in canonical form, let $s=agb^*$ with $g\in G$ and $a,b\in \PathSet$. By the definition of $\mathcal{G}$, the domain of the action of $s$ contains $\mathfrak{p}$ and so $b$ must be a prefix of $\mathfrak{p}$.

Now let $p_i$ be the length $i$ prefix of $\mathfrak{p}$. By the germ relation, we have
\[[agb^*, \mathfrak{p}]=[agb^*p_ip_i^*, \mathfrak{p}]\]
for all natural numbers $i$. Suppose that $|p_i|\geq |b|$ so that $p_i=bq$ for some $q\in \PathSet$. Then,
\[agb^*p_ip_i^* = agqp_i^* = ag(q)g|_qp_i^*.\]
Because $G$ is contracting, we may choose a natural number $i$ large enough such that $g|_{q}\in N$. So 
\[\overline{X}(\mathfrak{p})\subseteq \{[anp^*, \mathfrak{p}]: n\in N, a\in \PathSet, p\text{ is a prefix of }\mathfrak{p}\}.\]

Suppose that $[anp^*, \mathfrak{p}] \in \overline{X}(\mathfrak{p})$. Then, every open neighbourhood of $[anp^*, \mathfrak{p}]$ intersects $X$. In particular, for every prefix $pq$ of $\mathfrak{p}$, the open neighbourhood $[anp^*, pqX]$ intersect $X$. It follows that there exists some $[anp^*, pq\mathfrak{w}]\in [anp^*, pqX]$ such that $[anp^*, pq\mathfrak{w}]=[\id_{\source(p)}, pq\mathfrak{w}]$. But then there must be a prefix $pqw$ of $pq\mathfrak{w}$ such that the action of $\id_{\source(p)}$ and $anp^*$ agree on $pqw\Gamma^\omega$. It follows that
\begin{align*}
     pqw(pqw)^*&= anp^*pqw (pqw)^* \\
        &= anqw (pqw)^* \\
        &= an(qw)n|_{qw} (pqw)^*. \\
\end{align*}
As $pqw(pqw)^*$ and $an(qw)n|_{qw} (pqw)^*$ are each in canonical form, it must be the case that $n|_{qw} = \id_{\range(w)}$. Because the action of $n$ is length preserving, we have that $n(qw)=qw$ and thus $a = p$, establishing that $[anp^*, \mathfrak{p}]\in \{[pnp^*, \mathfrak{p}]: n\in N, p\text{ is a prefix of }\mathfrak{p}\}$. The above holds for every prefix $pq$ of $\mathfrak{p}$, so $anp^*(\mathfrak{p})=\mathfrak{p}$, establishing that $[anp^*, \mathfrak{p}]\in\mathcal{G}_{\mathfrak{p}}^{\mathfrak{p}}$. We conclude that
\[\overline{X}(\mathfrak{p})\subseteq \{[pnp^*, \mathfrak{p}]: n\in N, p\text{ is a prefix of }\mathfrak{p}\}\cap\mathcal{G}_{\mathfrak{p}}^{\mathfrak{p}}.\]

Finally, we show that $|\overline{X}(\mathfrak{p})|$ is finite. Suppose that $[p_inp_i^*,\mathfrak{p}]\in\mathcal{G}_{\mathfrak{p}}^{\mathfrak{p}}$ where $n\in N$ and $p_i$ is the length $i$ prefix of $\mathfrak{p}\in X$. Then $p_inp_i^*(\mathfrak{p})=\mathfrak{p}$. Let $p_{i+1}=p_ie$ for $e\in\EdgeSet$. Because $p_inp_i^*(\mathfrak{p})=\mathfrak{p}$, we have
\[[p_inp_i^*,\mathfrak{p}] = [p_inep_{i+1}^*, \mathfrak{p}] = [p_ien|_{e}p_{i+1}, \mathfrak{p}] = [p_{i+1}n|_{e}p_{i+1}, \mathfrak{p}].\]
Since $n|_e\in N$, it follows that for each natural number $i$,
\[\{[p_inp_i^*, \mathfrak{p}]: n\in N\}\cap \mathcal{G}_{\mathfrak{p}}^{\mathfrak{p}}\subseteq \{[p_{i+1}np_{i+1}^*, \mathfrak{p}]: n\in N\}\cap \mathcal{G}_{\mathfrak{p}}^{\mathfrak{p}}.\]
Also, for each natural number $i$ we have
\[|\{[p_inp_i^*, \mathfrak{p}]: n\in N\}|\leq |N|.\]
We conclude that because
\[\overline{X}(\mathfrak{p})\subseteq \{[pnp^*, \mathfrak{p}]: n\in N, p\text{ is a prefix of }\mathfrak{p}\}\cap \mathcal{G}_{\mathfrak{p}}^{\mathfrak{p}} = \bigcup_{i\in \mathbb{N}}\{[p_inp_i^*, \mathfrak{p}]: n\in N\}\cap \mathcal{G}_{\mathfrak{p}}^{\mathfrak{p}},\]
we must have that $|\overline{X}(\mathfrak{p})|\leq |N|<\infty$.
\end{proof}
The next corollary immediately follows from Proposition \ref{Prop: finite closure of X(w)} in tandem with \cite[Corollaries 4.6, 4.8]{brix2025hausdorffcoversnonhausdorffgroupoids}.
\begin{Corollary}
    The singular ideal of the complex Steinberg algebra $\mathbb{C}\mathcal{G}$ is $\{0\}$ if and only if the singular ideal of the reduced $C^*$-algebra $C^*_r(\mathcal{G})$ is $\{0\}$.
\end{Corollary}
Proposition \ref{Prop: finite closure of X(w)} along with \cite[Corollary 4.10]{brix2025hausdorffcoversnonhausdorffgroupoids} and \cite[Corollary 4.12]{clark2019simplicityalgebrasassociatednonhausdorff} also gives us the following.
\begin{Corollary}\label{Cor: cstar simple iff Steinberg simple}
    The complex Steinberg algebra $\mathbb{C}\mathcal{G}$ is simple if and only if the reduced $C^*$-algebra $C^*_r(\mathcal{G})$ is simple.
\end{Corollary}
\section{Examples}\label{Sec: Examples}
We begin with a self-similar groupoid which first appeared in Example 3.10 of \cite{SSGroupoidsWhitaker} where simplicity of the associated algebras was not considered. As previously noted, our conventions differ from those of Laca \textit{et al.} and what appears here is adjusted to align with our setting.

The self-similar groupoid $G$ of Example \ref{Ex: example 3.10 of whittaker} has one orbit, so it is natural to apply the techniques of Section \ref{Sec: Collapsing Groupoids}. Collapsing the self-similar groupoid $(G, \Sigma)$ results in a self-similar group previously considered in the literature.
\begin{Example}\label{Ex: example 3.10 of whittaker}
    Let $G$ be the groupoid acting on the paths of the graph $\Sigma$ of Figure \ref{fig: Sigma graph from Whittaker},
    \begin{figure}[ht!]
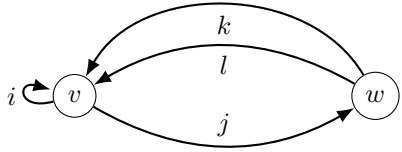

        \centering
        \tikz [>=Latex]{
        \graph [edge quotes={auto}]{ 
            "$v$" [at={(0,2)}, shape=circle, draw] ->[bend right,"$j$", thick]  "$w$" [at={(3,2)}, shape=circle, draw];
            "$w$" ->[bend right=60, "$k$", thick] "$v$";
            "$w$" ->[bend right, "$l$", thick] "$v$";
            "$v$" ->[loop left, "$i$", thick] "$v$";
        };}
        \caption{The graph $\Sigma$ of Example \ref{Ex: example 3.10 of whittaker}.}
        \label{fig: Sigma graph from Whittaker}
    \end{figure}
    generated by $a$ and $b$ whose actions are given by,
    \begin{center}
        \begin{tabular}{ l l }
         $a(ip)=lp,$ & $b(kp)=ip,$ \\
         $a(jq)=kb(q),$ & $b(lp)=ja(p)$, 
        \end{tabular}
\end{center}
for $p\in v\Sigma^*$ and $q\in w\Sigma^*$. It is easy to see that $G$ has a single orbit. Let $R=\{v\}$ and $F=\{\id_v, \id_u, a, a^{-1}\}$, and let $\alpha: G\rightarrow \faktor{G}{\psi_F}$ and $\beta:\Sigma^*\rightarrow \faktor{\Sigma^*}{\sim_F}$ be the groupoid and graph morphisms as specified in Section \ref{Sec: Collapsing Groupoids}. Let $\alpha(G)=L$ and the graph induced by $\beta(\Sigma^0\cup\Sigma^1)$ be $\Lambda$. Because $G$ has a single orbit, $L$ is a group and $\Lambda$ is the bouquet of two loops which we shall denote $\{e,f\}$. In fact, $L$ is a self-similar group by Proposition \ref{Prop: L,Lambda is self-similar}. The graph homomorphism $\beta: \Sigma\rightarrow \Lambda$ is specified by $\beta^{-1}(e)=\{i,l\}$ and $\beta^{-1}(f)=\{j,k\}$.

The set $v\Sigma^*$ acts as a set of representatives for the $\sim_F$ equivalence classes of $\Sigma^*$. We may then specify an identification $\gamma$ between $\faktor{\Sigma^*}{\sim_F}$ and $\lambda^*$ by specifying the image of $v\Sigma^*$. Let $\gamma$ satisfy $\gamma(\beta(pc))=\gamma(\beta(p))\beta(c)$ for all $pc\in v\Sigma^*$ with $c\in \EdgeSet$. The groupoid $G$ is generated by $\{a, a^{-1},b, b^{-1}\}$, and so $L$ is generated by $\{c=\alpha(a), c^{-1}=\alpha(a^{-1}), d=\alpha(b), d^{-1}=\alpha(b^{-1})\}$. The actions of $c$ and $d$ are given by,
    \begin{center}
        \begin{tabular}{ l l }
        $c(ep)=ep,$ & $d(ep)=fc(p),$\\ 
        $c(fp)=fd(p),$ & $d(fp)=ep,$
        \end{tabular}
    \end{center}
        for all $p\in \Lambda^*$. The set $\{c,d,c^{-1}, d^{-1}\}$ happens to be the standard generating set for a self-similar group known as the Basilica group \cite{basilicaSourcePaper} with its usual action on the bouquet of two loops. It is well known that the nucleus of $L$ is $\{\id, c, c^{-1}, d, d^{-1}, cd^{-1}, dc^{-1}\}$. The Moore diagram and the graph $\HH$ of the nucleus of $L$ are given in Figure \ref{Fig: moore diagram and H-graph of L.}.
    \begin{figure}[ht!]
        \centering
        \begin{tikzpicture}[->,>=stealth',shorten >=1pt,auto,node distance=4cm,
                thick,main node/.style={circle,draw,font=\Large\bfseries,red}]
        \node[main node] (id) {$\id$};
        \node (lghost) [left=4 of id] {};
        \node (rghost) [right=4 of id] {};
        \node (upghost) [above=4 of id] {};
        \node[main node] (c) [below=1.5 of lghost] {$c$};
        \node[main node] (cinv) [below=1.5 of rghost] {$c^{-1}$};
        \node[main node] (d) [above=1.5 of lghost] {$d$};
        \node[main node] (dinv) [above=1.5 of rghost] {$d^{-1}$};
        \node[main node] (cinvd) [right=1.5 of upghost] {$c^{-1}d$};
        \node[main node] (dinvc) [left=1.5 of upghost] {$d^{-1}c$};
        
        \path
            (c) edge [red] node [above] {$(e,e)$} (id)
                edge [bend left=15,red] node  {$(f,f)$} (d)
            (d) edge [bend left=15] node  {$(e,f)$} (c)
                edge node [below]  {$(f,e)$} (id)
            (id) edge [loop above,red] node  {$(e,e)$} (id)
                edge [loop below,red] node  {$(f,f)$} (id)
            (cinv) edge [red] node [above] {$(e,e)$} (id)
                edge [bend left=15,red] node  {$(f,f)$} (dinv)
            (dinv) edge [bend left=15] node  {$(f,e)$} (cinv)
                edge node [below] {$(e,f)$} (id)
            (dinvc) edge [bend left=15] node  {$(f,e)$} (cinvd)
                edge node [right] {$(e,f)$} (id)
            (cinvd) edge [bend left=15] node  {$(e,f)$} (dinvc)
                edge node [left] {$(f,e)$} (id);
            
        \end{tikzpicture}
        \caption{The Moore diagram (in black and red) and $\HH$(in red) of the nucleus of $L$.}
        \label{Fig: moore diagram and H-graph of L.}
    \end{figure}
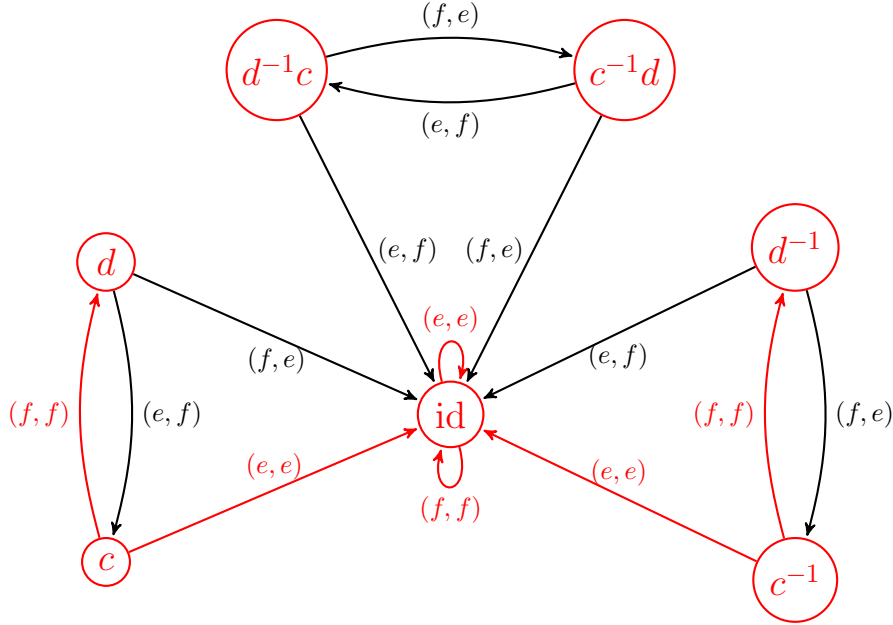
    Because there are no red cycles in Figure \ref{Fig: moore diagram and H-graph of L.}, except for the loops at the identity, there are no nontrivial recurring subgroups. Then the recurring subgroup condition of Corollary \ref{Cor: algebra quotient tight simple iff ker of recurring subgroup maps is 0} is trivially met and so $\faktor{KS_{(L, \Lambda)}}{{\TightIdeal}}$ is simple over every field $K$. It follows from Corollary \ref{Cor: groupoid algebra simple iff collapsed groupoid algebra simple} that $\faktor{KS_{(G, \Sigma)}}{{\TightIdeal}}$ is simple over every field $K$. We remark that the simplicity of the Steinberg algebra associated to the Basilica group was previously determined in \cite[Theorem 6.1]{SteinbergandNora2023}, and it is mentioned there that the ample groupoid associated to the Basilica group is Hausdorff.

    It is worth observing that $L$ being the Basilica group implies that the isotropy groups of $G$ are isomorphic as groups to the Basilica group. Thus it may surprise the reader that the Moore diagram of the nucleus of $G$, depicted in Figure 4 of \cite{SSGroupoidsWhitaker}, does not contain Figure \ref{Fig: moore diagram and H-graph of L.} as a subdiagram. We reconcile this oddity by recalling that the nucleus of a self-similar group or groupoid is a property of the action rather than a property of the group or groupoid itself, and though the action of $L$ on $\Lambda^*$ is built to sufficiently capture the essence of the action of $G$ on $\Sigma^*$, the structural differences of $\Lambda$ and $\Sigma$ are likely to blame for the apparent discrepancy between the nuclei of $L$ and $G$.
\end{Example}
Next we consider an example which illustrates the usefulness of Proposition \ref{Cor: groupoid algebra simple iff collapsed groupoid algebra simple} in greater generality than Example \ref{Ex: example 3.10 of whittaker}.
\begin{Example}\label{Ex: collapsing a groupoid onto a group}
    Let $H$ be any self-similar group acting on words over $\{a,b\}$. Identify $H$ with its action, a subset of $\textbf{Iso}(\{a,b\}^*)$. We define a self-similar groupoid $G$ acting on the graph $\Gamma$ shown in Figure \ref{Fig: lift of self-similar group graph} in the following way.

    Let $\beta: \Gamma^1:\rightarrow \{a,b\}$ be given by $\beta^{-1}(a)=\{e,g,i\}$ and $\beta^{-1}(b)=\{f,h,j\}$. Extend $\beta$ to $\PathSet$ by recursively defining $\beta(ep)=\beta(e)\beta(p)$ for $ep\in \PathSet$. For each $v\in \VertSet$, let the restriction of $\beta$ to $v\PathSet$ be denoted by $\beta_v$. Let $G$ be the subgroupoid of $\textbf{Iso}(\PathSet)$ consisting maps of the form $\beta_u^{-1}h\beta_v$, where $u,v\in \VertSet$ and $h\in H$.

    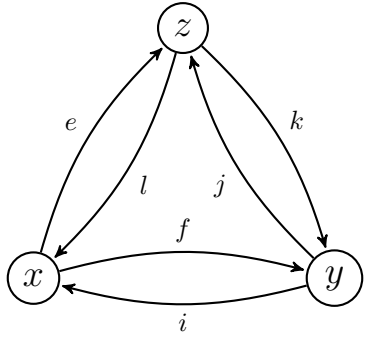
\begin{figure}[ht!]
        \centering
        \begin{tikzpicture}[->,>=stealth',shorten >=1pt,auto,node distance=4cm,
                thick,main node/.style={circle,draw,font=\Large\bfseries}]
        \node[main node] (x) {$x$};
        \node (ghost) [right=1.5 of x] {};
        \node[main node] (y) [right=1.5 of ghost] {$y$};
        \node[main node] (z) [above=2.85 of ghost] {$z$};
        
        \path
            (x) edge [bend left=15] node  {$e$} (z)
                edge [bend left=15] node  {$f$} (y)
            (z) edge [bend left=15] node  {$k$} (y)
                edge [bend left=15] node  {$l$} (x)
            (y) edge [bend left=15] node  {$i$} (x)
                edge [bend left=15] node  {$j$} (z);
        \end{tikzpicture}
        \caption{The graph $\Gamma$ of Example \ref{Ex: collapsing a groupoid onto a group}.}
        \label{Fig: lift of self-similar group graph}
    \end{figure}
It is easy to verify that $G$ is a self-similar groupoid. Indeed, let $\beta_u^{-1}h\beta_v\in G$ and $c\in\EdgeSet$ and suppose that $\beta_u^{-1}h\beta_v(c)=d$. Then,
\begin{align*}
    \beta_u^{-1}h\beta_v(cp) &= \beta_u^{-1}h(\beta_v(c)\beta_{\range(c)}(p))\\
                        &= \beta_u^{-1}(h(\beta_v(c)\beta_{\range(c)}(p))) \\
                        &= d\beta_{\range(d)}^{-1}(h|_{\beta_v(c)}(\beta_{\range(c)}(p))) \\
                        &= d\beta_{\range(d)}^{-1}h|_{\beta_v(c)}\beta_{\range(c)}(p)
\end{align*}
for all $p\in \range(c)\PathSet$. Let us analyze the singular and tight ideals of $KS_{(G, \Gamma)}$ by applying the techniques of Section \ref{Sec: Collapsing Groupoids}.

The subset $F=\{\id_x, \id_y, \id_v, \beta_{x}^{-1}\id\beta_{y}, \beta_{y}^{-1}\id\beta_{x}, \beta_{x}^{-1}\id\beta_{z}, \beta_{z}^{-1}\id\beta_{x}\}$ of $G$ establishes that $G$ has a single orbit. Thus, collapsing $(G, \Gamma)$ in the manner of Section \ref{Sec: Collapsing Groupoids} results in a self-similar group. It is clear from the construction of $G$ that using the subset $F$ to specify the collapsing maps, we get exactly $(H,\{a,b\})$ as the collapsed groupoid (where $\{a,b\}$ is identified with a bouquet of loops). Applying Corollary \ref{Cor: groupoid algebra simple iff collapsed groupoid algebra simple}, the algebra $\faktor{KS_{(G, \Gamma)}}{\TightIdeal}$ is simple if and only if the algebra $\faktor{KS_{(H,\{a,b\})}}{\TightIdeal}$ is simple. 
\end{Example}
We end with an example whose construction is inspired by the multispinal groups of \cite{SteinbergandNora2023}. In particular, this groupoid contains both the Grigorchuk group and Grigorchuk-Erschler group as self-similar subgroups.
\begin{Example}\label{Ex: Grigorchuk meets Grigorchuk-Erschler MS}
Let $M$ be the self-similar groupoid acting on the graph $\Gamma$ in Figure \ref{fig: Multispinal groupoid graph} generated by $\{a_x,b_x,c_x,d_x,a_y,b_y,c_y,d_y\}$ whose actions and restrictions are given by:
\begin{center}
\begin{tabular}{ l l | l l }
 $a_x(ip)=jp,$ & $a_x(jp)=ip,$ & $a_x(eq)=fq,$ & $a_x(fq)=eq$,\\ 
 $b_x(ip)=ia_x(p),$ & $b_x(jp)=jc_x(p),$ & $b_x(eq)=eq,$ & $b_x(fq)=fq$,\\
 $c_x(ip)=ia_x(p),$ & $c_x(jp)=jd_x(p),$ & $c_x(eq)=eq,$ & $c_x(fq)=fq$,\\ 
 $d_x(ip)=ip,$ & $d_x(jp)=jb_x(p),$ & $d_x(eq)=eq,$ & $d_x(fq)=fq$,\\  
\hline
 $a_y(mq)=nq,$ & $a_y(nq)=mq,$ & $a_y(gp)=hp$ & $a_y(hp)=gp$,\\ 
 $b_y(mq)=ma_y(q),$ & $b_y(nq)=nb_y(q),$ & $b_y(gp)=gp$ & $b_y(hp)=hp$,\\
 $c_y(mq)=ma_x(q),$ & $c_y(nq)=nd_y(q),$ & $c_y(gp)=gp$ & $c_y(hp)=hp$,\\ 
 $d_y(mq)=mq,$ & $d_y(nq)=nb_y(q),$ & $d_y(gp)=gp$ & $d_y(hp)=hp$, 
\end{tabular}
\end{center}
for each $p\in x\Gamma^*$ and $q\in y\Gamma^*$.

    \begin{figure}[ht!]
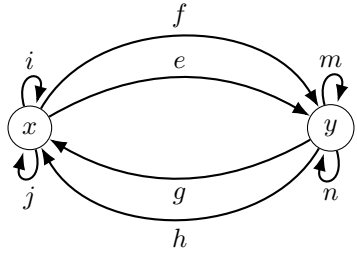

        \centering
        \tikz [>=Latex]{
        \graph [edge quotes={auto}]{ 
            "$x$" [at={(0,2)}, shape=circle, draw] ->[bend left,"$e$", thick]  "$y$" [at={(3,2)}, shape=circle, draw];
            "$x$" ->[bend left=60, "$f$", thick] "$y$";
            "$y$" ->[bend left, "$g$", thick] "$x$";
            "$y$" ->[bend left=60, "$h$", thick] "$x$";
            "$x$" ->[loop above, "$i$", thick] "$x$";
            "$x$" ->[loop below, "$j$", thick] "$x$";
            "$y$" ->[loop above, "$m$", thick] "$y$";
            "$y$" ->[loop below, "$n$", thick] "$y$";
        };}
        \caption{The graph $\Gamma$ of Example \ref{Ex: Grigorchuk meets Grigorchuk-Erschler MS}.}
        \label{fig: Multispinal groupoid graph}
    \end{figure}
    The keen eye will recognize two things about these generators: first, the equations in the upper left quadrant define the Grigorchuk group acting on the subgraph containing $x,i,j$; and second, the equations in the lower left quadrant define the Grigorchuk-Erschler group on the subgraph containing $y,m,n$. The equations in the right two quadrants 'weave' together the self-similarity of these two groups in a fairly simple way, by letting each restriction from one group into the other go to the respective identity.

    The self-similar groupoid $M$ is contracting and has nucleus,
    \[N=\{\id_x, a_x,b_x,c_x,d_x, \id_y, a_y,b_y,c_y,d_y\}.\] 
    Figure \ref{Fig: MS H graph} displays the subgraph $\HH$ of the Moore diagram of the nucleus, as described in Section \ref{Sec: simplicity algorithm}.
    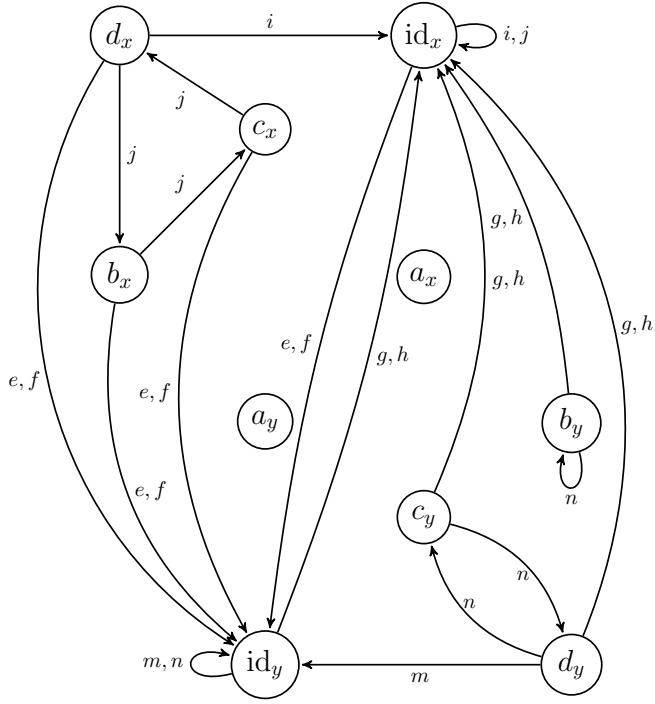
\begin{figure}[ht!]
        \centering
        \scalebox{0.8}{
        \begin{tikzpicture}[->,>=stealth',shorten >=1pt,auto,node distance=4cm,
                thick,main node/.style={circle,draw,font=\Large\bfseries}]
        \node[main node] (iy)  {$\id_y$};
        \node[main node] (dy) [right=4 of iy] {$d_y$};
        \node[main node] (cy) [above left=2.5 of dy] {$c_y$};
        \node[main node] (by) [above=3 of dy] {$b_y$};
        
        \node[main node] (ix) [above =7 of cy] {$\id_x$};
        \node[main node] (dx) [left=4 of ix] {$d_x$};
        \node[main node] (bx) [below=3 of dx] {$b_x$};
        \node[main node] (cx) [above right=2.5 of bx] {$c_x$};

        \node[main node] (ax) [below=3 of ix] {$a_x$};
        \node[main node] (ay) [above=3 of iy] {$a_y$};

        \path
            (bx) edge node {$j$} (cx)
            (cx) edge node {$j$} (dx)
            (dx) edge node {$j$} (bx)
                edge node {$i$} (ix)
            (ix) edge [loop right] node {$i,j$} (ix)
            
            (by) edge [loop below] node {$n$} (by)
            (cy) edge [bend left] node [below] {$n$} (dy)
            (dy) edge [bend left] node [above] {$n$} (cy)
                edge node {$m$} (iy)
            (iy) edge [loop left] node {$m,n$} (iy)

            (ix) edge [bend right=7] node [left] {$e,f$} (iy)
            (bx) edge [bend right] node [right] {$e,f$} (iy)
            (cx) edge [bend right] node [left] {$e,f$} (iy)
            (dx) edge [bend right=45] node [left] {$e,f$} (iy)

            (iy) edge [bend right=7] node [right] {$g,h$} (ix)
            (by) edge [bend right=15] node [left] {$g,h$} (ix)
            (cy) edge [bend right=23] node [right] {$g,h$} (ix)
            (dy) edge [bend right=35] node [right] {$g,h$} (ix)
            ;
        \end{tikzpicture}
        }
        \caption{The graph $\HH$ associated to $M$.}
        \label{Fig: MS H graph}
    \end{figure}

    The graph $\Delta$ of Theorem \ref{Thm: Ess neq tight iff H_p and minimal vertices of synchronized verts of delta} is shown in Figure \ref{fig: MS graph Delta}, where we employ the unconventional notation $\{\id,b,c,d\}_x$ to denote $\{\id_x,b_x,c_x,d_x\}$ in the interest of readability.
    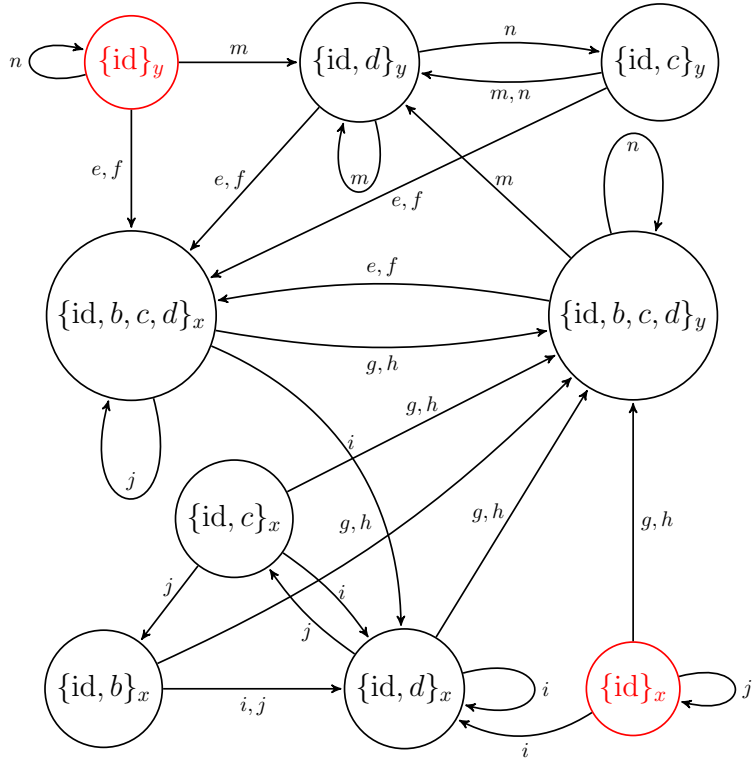
\begin{figure}[ht!]
        \centering
        \scalebox{0.8}{
        \begin{tikzpicture}[->,>=stealth',shorten >=1pt,auto,node distance=4cm,
                thick,main node/.style={circle,draw,font=\Large\bfseries}]
        \node[main node] (ibcdx) {$\{\id,b, c, d\}_x$};
        \node[main node] (ibcdy) [right=5.5 of ibcdx] {$\{\id,b, c, d\}_y$};
        \node[main node, red] (ix) [below=4 of ibcdy] {$\{\id\}_x$};
        \node[main node] (idx) [left=2 of ix] {$\{\id,d\}_x$};
        \node[main node] (ibx) [left=3 of idx] {$\{\id,b\}_x$};
        \node[main node] (icx) [above left=2 of idx] {$\{\id,c\}_x$};

        \node[main node, red] (iy) [above=2 of ibcdx] {$\{\id\}_y$};
        \node[main node] (idy) [right=2 of iy] {$\{\id,d\}_y$};
        \node[main node] (icy) [right=3 of idy] {$\{\id,c\}_y$};

        \path
            (ix) edge [loop right] node [right] {$j$} (ix)
                edge [bend left]  node {$i$} (idx)
            (idx) edge [loop right] node [right] {$i$} (idx)
                edge [bend left=10] node [below] {$j$} (icx)
            (icx) edge [bend left=10] node [right] {$i$} (idx)
                edge node [above] {$j$} (ibx)
            (ibx) edge node [below] {$i,j$} (idx)
            
            (iy) edge node [left] {$e,f$} (ibcdx)
            (idy) edge node [left] {$e,f$} (ibcdx)
            (icy) edge node [below] {$e,f$} (ibcdx)

            (ix) edge  node [right] {$g,h$} (ibcdy)
            (idx) edge  node [left] {$g,h$} (ibcdy)
            (icx) edge  node [above] {$g,h$} (ibcdy)
            (ibx) edge [bend right = 10] node  {$g,h$} (ibcdy)

            (ibcdx) edge [bend right=10] node [below] {$g,h$} (ibcdy)
                  edge [loop below] node [above] {$j$} (ibcdx)
                  edge [bend left=33] node [above] {$i$} (idx)
            (ibcdy) edge [bend right=10] node [above] {$e,f$} (ibcdx)
                   edge [loop above] node [below] {$n$} (ibcdy)
                   edge node [right] {$m$} (idy)
            
            (iy) edge [loop left] node {$n$} (iy)
                edge node {$m$} (idy)
            (idy) edge [bend left=10] node {$n$} (icy)
                edge [loop below] node [above] {$m$} (idy)
            (icy) edge [bend left=10] node {$m,n$} (idy);
        \end{tikzpicture}
        }
        \caption{The graph $\Delta$ associated to $M$ with minimal vertices shown in black.}
        \label{fig: MS graph Delta}
    \end{figure}
    We refer the reader to Figures 1 and 2 of \cite{gardella2025simplicitycalgebrascontractingselfsimilar} for comparison with the graph $\Delta$ associated to the Grigorchuk group and the Grigorchuk-Erschler group, respectively. 
    
    The graph $\Gamma$ in Figure \ref{fig: Multispinal groupoid graph} is strongly connected and each vertex has out-degree 4, so by Corollary \ref{Cor: CongruenceFreeAppliedS_G with nonempty graph}, the inverse semigroup $S_{(G, \Gamma)}$ is congruence-free. It follows from Proposition \ref{Prop: synchronized iff minimal when strongly connected} that the minimal vertices of $\Delta$ coincide with the synchronized vertices of $\Delta$. Observe that every vertex is in the minimal component of $\Delta$ except for $\{\id_x\}$ and $\{\id_y\}$, as indicated by the vertex coloring.

    Using Figure \ref{Fig: MS H graph}, one may check that $H_{jjj}=\{\id_x,b_x, c_x, d_x\}$, $H_n=\{\id_y, b_y\}$, and $H_{nn}=\{\id_y, b_y, c_y, d_y\}$ are the only non-trivial recurring subgroups. The nontrivial intersections of $H_{jjj}$ give rise to the following relevant kernels:
    \begin{align*}
        \ker \pi_{jjj, eg} &= \{k_{1}\id_x+ k_{b}b_x+k_cc_x+k_dd_x: k_{1}+ k_{b}+k_c+k_d=0\}, \\
        \ker \pi_{jjj, i} &= \{k_{1}\id_x+ k_{b}b_x+k_cc_x+k_dd_x: k_{1}+k_d=0= k_{b}+k_c\}, \\
        \ker \pi_{jjj, ji} &= \{k_{1}\id_x+ k_{b}b_x+k_cc_x+k_dd_x: k_{1}+k_c=0= k_{b}+k_d\}, \\
        \ker \pi_{jjj, jji} &= \{k_{1}\id_x+ k_{b}b_x+k_cc_x+k_dd_x: k_{1}+k_b=0= k_{c}+k_d\}.
    \end{align*}
    It follows from \ref{Prop: a in essIdeal iff a satisfies synchronized verts} that when $K$ has characteristic 2, $KH_{jjj}\cap \SingularIdeal$ is spanned by $\id_x+b+c+d$, and for any other characteristic, $KH_{jjj}\cap \SingularIdeal=\{0\}$.

    The nontrivial intersections of $H_{nn}$ give rise to the following kernels:
    \begin{align*}
        \ker \pi_{nn, ge} &= \{k_{1}\id_y+ k_{b}b_y+k_cc_y+k_dd_y: k_{1}+k_c+ k_{b}+k_d=0\}, \\
        \ker \pi_{nn, m} &= \{k_{1}\id_y+ k_{b}b_y+k_cc_y+k_dd: k_{1}+ k_{d}=0=k_b+k_c=0\}, \\
        \ker \pi_{nn, nm} &= \{k_{1}\id_y+ k_{b}b_y+k_cc_y+k_dd_y: k_{1}+k_c=0= k_{b}+k_d\}.
    \end{align*}
    It follows from \ref{Prop: a in essIdeal iff a satisfies synchronized verts} that the span of $\id_{y}+b_y-c_y-d_y$ is contained in $KH_{nn}\cap\SingularIdeal$ regardless of the characteristic of $K$.

    We conclude that regardless of the characteristic of $K$, $\faktor{KS_{(G, \Gamma)}}{\TightIdeal}$ is not simple by Corollary \ref{Cor: algebra quotient tight simple iff ker of recurring subgroup maps is 0}. Additionally, by Corollary \ref{Cor: cstar simple iff Steinberg simple}, the reduced $C^*$-algebra associated to $(G, \Gamma)$ is not simple.
    
    We remark that because the maximal subgroups of $(G, \Gamma)$ so closely resemble the Grigorchuk group or the Grigorchuk-Erschler group (indeed, the maximal subgroups are these two groups with extended actions), the analysis of the simplicity of the algebras associated to $(G, \Gamma)$ essentially combines the existing simplicity analysis of the Grigorchuk group in \cite[Figure 1]{gardella2025simplicitycalgebrascontractingselfsimilar} (see also \cite[Section 7]{SteinbergandNora2023}) and the simplicity analysis of the Grigorchuk-Erschler group in \cite[Figure 2]{gardella2025simplicitycalgebrascontractingselfsimilar}. However, there were choices made in how the actions of those two groups were extended and woven together to form $G$. Had different decisions been made, it would be possible to introduce recurring subgroups of $G$ which are not present in either the Grigorchuk group nor the Grigorchuk-Erschler group.
\end{Example}
\printbibliography

@article{AshHallGraphISGs,
  author  = {Ash, C. J. and Hall, T. E.},
  title   = {Inverse semigroups on graphs},
  journal = {Semigroup Forum},
  volume  = {11},
  pages   = {140-145},
  year    = {1975},
  issn    = {1432-2137}
}

@article{Baird_1975,
  title   = {Congruence-free inverse semigroups with zero},
  volume  = {20},
  number  = {1},
  journal = {Journal of the Australian Mathematical Society},
  author  = {Baird, G. R.},
  year    = {1975},
  pages   = {110-114}
}

@article{basilicaSourcePaper,
  author  = {Rostislav Grigorchuk and Andrzej Zuk},
  year    = {2002},
  month   = {02},
  pages   = {223-245},
  title   = {On a torsion-Free weakly branch group defined by a three state automaton.},
  volume  = {12},
  journal = {International Journal of Algebra and Computation}
}

@misc{brix2025hausdorffcoversnonhausdorffgroupoids,
  title         = {On Hausdorff covers for non-Hausdorff groupoids},
  author        = {Kevin Aguyar Brix and Julian Gonzales and Jeremy B. Hume and Xin Li},
  year          = {2025},
  eprint        = {2503.23203},
  archiveprefix = {arXiv},
  primaryclass  = {math.OA}
}

@article{brownloweWhittaker2023kkdualityselfsimilargroupoidactions,
  title     = {KK-duality for self-similar groupoid actions on graphs},
  author    = {Nathan Brownlowe and Alcides Buss and Daniel Gon\c{c}alves and Jeremy B. Hume and Aidan Sims and Michael F. Whittaker},
  journal   = {Transactions of the American Mathematical Society},
  volume    = {377},
  year      = {2024},
  pages     = {5513--5560},
  doi       = {10.1090/tran/9183},
  url       = {https://doi.org/10.1090/tran/9183},
  publisher = {American Mathematical Society},
  address   = {Providence, RI},
  issn      = {0002-9947},
}

@article{clark2012groupoidgeneralizationleavittpath,
  author    = {Lisa Orloff Clark and Cynthia Farthing and Aidan Sims and Mark Tomforde},
  title     = {A groupoid generalisation of Leavitt path algebras},
  journal   = {Semigroup Forum},
  year      = {2014},
  volume    = {89},
  number    = {3},
  pages     = {501--517},
  doi       = {10.1007/s00233-014-9594-z},
  url       = {https://doi.org/10.1007/s00233-014-9594-z},
  issn      = {1432-2137},
}

@article{clark2019simplicityalgebrasassociatednonhausdorff,
  author    = {Lisa Orloff Clark and Ruy Exel and Enrique Pardo and Aidan Sims and Charles Starling},
  title     = {Simplicity of algebras associated to non-Hausdorff groupoids},
  journal   = {Transactions of the American Mathematical Society},
  volume    = {372},
  pages     = {3669--3712},
  year      = {2019},
  url       = {https://doi.org/10.1090/tran/7840},
  publisher = {American Mathematical Society},
}

@article{clark2022Steinbergalgebraapproachetale,
  title     = {A Steinberg algebra approach to \'etale groupoid $C^{*}$-algebras},
  author    = {Lisa Orloff Clark and Joel Zimmerman},
  journal   = {Journal of Operator Theory},
  volume    = {91},
  number    = {2},
  pages     = {349--371},
  year      = {2024}
}

@article{Cuntz1980ACO,
  title     = {A class of $C^*$-algebras and topological Markov chains},
  author    = {Joachim J. R. Cuntz and Wolfgang Krieger},
  journal   = {Inventiones mathematicae},
  year      = {1980},
  volume    = {56},
  pages     = {251-268}
}

@article{exel2008inversesemigroupscombinatorialcalgebras,
  title     = {Inverse semigroups and combinatorial $C^*$-algebras},
  author    = {Ruy Exel},
  journal   = {Bulletin of the Brazilian Mathematical Society, New Series},
  year      = {2007},
  volume    = {39},
  pages     = {191-313},
  url       = {https://api.semanticscholar.org/CorpusID:12287830}
}

@article{ExelEssential_2022,
  title     = {Characterizing groupoid $C^*$-algebras of non-Hausdorff \'etale groupoids},
  isbn      = {9783031055133},
  issn      = {1617-9692},
  journal   = {Lecture Notes in Mathematics},
  publisher = {Springer International Publishing},
  author    = {Ruy Exel and David R. Pitts},
  year      = {2022}
}

@article{EXELPardoAlgebras,
title = {Self-similar graphs, a unified treatment of Katsura and Nekrashevych $C^*$-algebras},
journal = {Advances in Mathematics},
volume = {306},
pages = {1046-1129},
year = {2017},
issn = {0001-8708},
doi = {https://doi.org/10.1016/j.aim.2016.10.030},
url = {https://www.sciencedirect.com/science/article/pii/S0001870816314396},
author = {Ruy Exel and Enrique Pardo},
}

@article{Farthing2005,
  title   = {Higher-rank graph $C^*$-algebras: An inverse semigroup and groupoid approach},
  author  = {Cynthia Farthing and Paul S. Muhly and Trent Yeend},
  year    = {2005},
  journal = {Semigroup Forum},
  volume  = {71},
  number  = {2},
  pages   = {159-187},
  url     = {https://doi.org/10.1007/s00233-005-0512-2}
}

@article{gardella2025simplicitycalgebrascontractingselfsimilar,
  title         = {Simplicity of $C^*$-algebras of contracting self-similar groups},
  author        = {Eusebio Gardella and Volodymyr Nekrashevych and Benjamin Steinberg and Alina Vdovina},
  year          = {2025},
  journal       = {Communications in Mathematical Physics},
  volume        = {406},
  number        = {10},
  pages         = {251},
url             = { https://doi.org/10.1007/s00220-025-05411-5}
}

@article{grigorchuk,
  author  = {Rostislav Grigorchuk},
  title   = {{Burnside problem on periodic groups}},
  journal = {Funktsional. Anal. i Prilozhen.},
  volume  = {14},
  pages   = {53-54},
  year    = {1980}
}

@article{HazratPaskSimsSiera,
  title   = {An algebraic analogue of Exel-Pardo $C^*$-algebras},
  author  = {Roozbeh Hazrat and David Pask and Adam Sierakowski and Aidan Sims},
  pages   = {877-909},
  volume  = {24},
  number  = {4},
  journal = {Algebras and Representation Theory},
  year    = {2021}
}

@book{howie1995fundamentals,
  title     = {Fundamentals of Semigroup Theory},
  author    = {John M. Howie},
  isbn      = {9780198511946},
  series    = {London Mathematical Society Monographs},
  year      = {1995},
  publisher = {Oxford University Press}
}

@misc{lawson2020primer,
  title         = {Primer on inverse semigroups I},
  author        = {Mark Lawson},
  year          = {2020},
  eprint        = {2006.01628},
  archiveprefix = {arXiv},
  primaryclass  = {math.GR}
}

@article{LawsonPerrotMonoid,
  author  = {Mark Lawson},
  title   = {{A correspondence between a class of monoids and self-similar group actions I}},
  journal = {Semigroup Forum},
  volume  = {76},
  number  = {3},
  pages   = {489-517},
  year    = {2008},
  issn    = {1432-2137}
}

@article{Leavitt,
  title   = {The Leavitt path algebra of a graph},
  journal = {Journal of Algebra},
  volume  = {293},
  number  = {2},
  pages   = {319-334},
  year    = {2005},
  issn    = {0021-8693},
  author  = {Gene Abrams and Gonzalo {Aranda Pino}}
}

@article{Munn_Fundamental_ISG,
  author  = {W. D. Munn},
  title   = {{Fundamental Inverse Semigroups}},
  journal = {The Quarterly Journal of Mathematics},
  volume  = {21},
  number  = {2},
  pages   = {157-170},
  year    = {1970},
  month   = {06},
  issn    = {0033-5606}
}

@inproceedings{munn1979simple,
  author    = {W. D. Munn},
  title     = {Simple contracted semigroup algebras},
  booktitle = {Proceedings of the Conference on Semigroups in Honor of Alfred H. Clifford (Tulane Univ., New Orleans, La., 1978)},
  pages     = {35--43},
  publisher = {Tulane University},
  address   = {New Orleans, LA},
  year      = {1979}
}

@book{nekrashevych2005self,
  title     = {Self-Similar Groups},
  author    = {Volodymyr Nekrashevych},
  isbn      = {9780821838310},
  lccn      = {2005048021},
  series    = {Mathematical surveys and monographs},
  year      = {2005},
  publisher = {American Mathematical Society}
}

@article{nekrashevych2015growthetalegroupoidssimple,
    author = {Volodymyr Nekrashevych},
    title = {Growth of \'etale groupoids and simple algebras},
    journal = {International Journal of Algebra and Computation},
    volume = {26},
    number = {02},
    pages = {375-397},
    year = {2016},
    doi = {10.1142/S0218196716500156},
    URL = { https://doi.org/10.1142/S0218196716500156
    }
}

@article{Nekrashevych2009Alg,
author = {Volodymyr Nekrashevych},
year = {2009},
month = {05},
pages = {59--123},
title = {$C^*$-algebras and self-similar groups},
volume = {2009},
journal = {Journal F\"ur Die Reine Und Angewandte Mathematik},
doi = {10.1515/CRELLE.2009.035}
}

@book{Paterson1998GroupoidsIS,
  title     = {Groupoids, inverse semigroups, and their operator algebras},
  author    = {Alan L. T. Paterson},
  year      = {1999},
  isbn      = {9780817640514},
  series    = {Progress in Mathematics},
  publisher = {Birkh\"auser},
  volume    = {170}
}

@article{Paterson2002,
  issn      = {03794024, 18417744},
  url       = {http://www.jstor.org/stable/24715591},
  author    = {Alan L.T. Paterson},
  journal   = {Journal of Operator Theory},
  number    = {3},
  pages     = {645-662},
  publisher = {Theta Foundation},
  title     = {Graph inverse semigroups, groupoids and their $C^*$-algebras},
  urldate   = {2025-10-04},
  volume    = {48},
  year      = {2002}
}

@article{pino2011kumjianpaskalgebrashigherrankgraphs,
  title         = {Kumjian-Pask algebras of higher-rank graphs},
  author        = {Gonzalo Aranda Pino and John Clark and Astrid an Huef and Iain Raeburn},
  year          = {2013},
  journal   = {Transactions of the American Mathematical Society},
  volume    = {365},
  number     = {7},
  pages     = {3613--3641},
}

@book{renaultcstaralgebra,
  isbn      = {9783540392187},
  publisher = {Springer},
  series    = {Lecture notes in mathematics},
  volume    = {793},
  title     = {A groupoid approach to $C^*$-algebras},
  year      = {1980},
  author    = {Renault, Jean},
  address   = {Berlin},
}

@article{Robertson_2007,
  title     = {Simplicity of $C^*$-algebras associated to higher-rank graphs},
  volume    = {39},
  issn      = {0024-6093},
  url       = {http://dx.doi.org/10.1112/blms/bdm006},
  doi       = {10.1112/blms/bdm006},
  number    = {2},
  journal   = {Bulletin of the London Mathematical Society},
  publisher = {Wiley},
  author    = {David Robertson and Aidan Sims},
  year      = {2007},
  pages     = {337-344}
}

@article{SSGroupoidsWhitaker,
  title   = {Equilibrium states on operator algebras associated to self-similar actions of groupoids on graphs},
  journal = {Advances in Mathematics},
  volume  = {331},
  pages   = {268-325},
  year    = {2018},
  issn    = {0001-8708},
  author  = {Marcelo Laca and Iain Raeburn and Jacqui Ramagge and Michael F. Whittaker}
}

@article{STEINBERG2010689,
  title   = {A groupoid approach to discrete inverse semigroup algebras},
  journal = {Advances in Mathematics},
  volume  = {223},
  number  = {2},
  pages   = {689-727},
  year    = {2010},
  issn    = {0001-8708},
  author  = {Benjamin Steinberg}
}

@article{Steinbergandnora2020simplicityinversesemigroupetale,
  title   = {Simplicity of inverse semigroup and \'etale groupoid algebras},
  journal = {Advances in Mathematics},
  volume  = {380},
  pages   = {107611},
  year    = {2021},
  issn    = {0001-8708},
  author  = {Benjamin Steinberg and N\'ora Szak\'acs}
}

@article{SteinbergandNora2023,
  title   = {On the simplicity of Nekrashevych algebras of contracting self-similar groups},
  author  = {Benjamin Steinberg and N\'{o}ra Szak\'{a}cs},
  pages   = {1391-1428},
  volume  = {386},
  number  = {3},
  issn    = {1432-1807},
  journal = {Mathematische Annalen},
  year    = {2023}
}

@article{Trotter1974,
  title   = {Congruence-free inverse semigroups},
  volume  = {9},
  number  = {1},
  journal = {Semigroup Forum},
  author  = {Trotter, P. G.},
  year    = {1974},
  pages   = {109--116}
}

@misc{martinez2025algebraicsingularfunctionsdense,
      title={Algebraic singular functions are not always dense in the ideal of $C^*$-singular functions}, 
      author={Diego Mart\'inez and N\'ora Szak\'acs},
      year={2025},
      eprint={2510.01947},
      archivePrefix={arXiv},
      primaryClass={math.OA},
      url={https://arxiv.org/abs/2510.01947}, 
}
\end{document}